\documentclass[leqno,12pt]{article}

\usepackage[all]{xy}
\usepackage{amsthm}
\usepackage{amssymb} 
\usepackage{amsmath,amscd} 
\usepackage{mathrsfs}
\usepackage{color}
\usepackage{enumerate}
\usepackage{longtable}

\def\frkm{{\mathfrak{m}}}

\def\depth{{\mathop{\mathrm{depth}}}}

\def\GrMod{\operatorname{\mathsf{GrMod}}}

\def\turn!{\textup{!`}}

\def\op{\textup{op}}

\def\soc{\operatorname{soc}}

\def\pd{\mathop{\mathrm{pd}}\limits}
\def\grpd{\mathop{\mathrm{gr.pd}}\limits}
\def\injdim{\mathop{\mathrm{id}}\limits}
\def\grinjdim{\mathop{\mathrm{gr.id}}\limits}

\def\grCM{\operatorname{\mathsf{CM}}^{\Bbb{Z}}}

\def\grmod{\operatorname{mod}^{\Bbb{Z}}}

\def\mrb{\mathrm{b}}

\def\Sing{\operatorname{Sing}}

\def\grSing{\operatorname{Sing}^{\ZZ}}

\def\grSpec{\operatorname{Spec^{\Bbb{Z}}}}

\def\stabCM{\operatorname{\underline{\mathsf{CM}}}}
\def\stabgrCM{\operatorname{\underline{\mathsf{CM}}}^{\Bbb{Z}}}

\def\stabgrmod{\operatorname{\underline{mod}}^{\Bbb{Z}}}

\def\thick{\mathop{\mathsf{thick}}\nolimits}

\def\kk{{\mathbf k}}

\def\NN{{\Bbb N}}

\def\ZZ{{\Bbb Z}}

\def\cA{{\cal A}}

\def\cH{{\cal H}}

\def\sfq{{\mathsf{q}}}

\def\sfC{{\mathsf{C}}}
\def\sfD{{\mathsf{D}}}

\def\sfI{{\mathsf{I}}}

\def\sfK{{\mathsf{K}}}

\def\sfT{{\mathsf{T}}}
\def\sfU{{\mathsf{U}}}
\def\sfV{{\mathsf{V}}}

\def\tuD{{\textup{D}}}

\def\tuH{{\textup{H}}}

\def\tuT{{\textup{T}}}

\def\frki{{\mathfrak{i}}}

\def\frkp{{\mathfrak{p}}}
\def\frkq{{\mathfrak{q}}}

\def\frks{{\mathfrak{s}}}

\def\id{\operatorname{id}}
\def\op{\operatorname{op}}

\def\mod{\operatorname{mod}}
\def\Mod{\operatorname{Mod}}
\def\GrMod{\operatorname{Mod}^{\mathbb{Z}}}

\def\Ker{\mathop{\mathrm{Ker}}\nolimits}

\def\proj{\operatorname{proj}}

\def\grproj{\operatorname{proj}^{\mathbb{Z}}}

\def\GrInj{\operatorname{Inj}^{\mathbb{Z}}}

\def\Proj{\operatorname{Proj}} 
\def\Inj{\operatorname{Inj}}

\def\grInj{\operatorname{Inj}^{\mathbb{Z}}}

\def\add{\operatorname{add}}

\def\Coker{\operatorname{Cok}}

\def\Hom{\operatorname{Hom}}

\def\cpxgrHom{\operatorname{HOM}^{\bullet}}
\def\grHom{\operatorname{HOM}}

\def\End{\operatorname{End}}

\def\Ext{\operatorname{Ext}}
\def\grExt{\operatorname{EXT}}

\def\cpxHom{\operatorname{Hom}^{\bullet}}

\def\gldim{\operatorname{gldim}}

\newcommand{\RHom}{\operatorname{\Bbb{R}Hom}}
\def\grRHom{\operatorname{\Bbb{R}HOM}}
\newcommand{\lotimes}{\otimes^{\Bbb{L}}}

\newcommand{\cone}{\operatorname{\mathsf{cn}}}

\def\RHom{\operatorname{\mathbb{R}Hom}}
\def\grRHom{\operatorname{\mathbb{R}HOM}}

\newtheorem{lemma}{Lemma}[section]
\newtheorem{proposition}[lemma]{Proposition}
\newtheorem{theorem}[lemma]{Theorem}
\newtheorem{corollary}[lemma]{Corollary}

\theoremstyle{definition}
\newtheorem{remark}[lemma]{Remark}
\newtheorem{example}[lemma]{Example}

\newtheorem{definition}[lemma]{Definition}

\newtheorem{question}[lemma]{Question}

\theoremstyle{remark}

\definecolor{ao}{rgb}{0.0, 0.5, 0.0}

\title{
The Happel functor and 
homologically well-graded Iwanaga-Gorenstein algebras
}

\author{Hiroyuki Minamoto and Kota Yamaura}

\begin{document}

\maketitle
\begin{abstract}
Happel constructed a fully faithful  functor $\cH :\sfD^{\mrb}(\mod \Lambda) \to \stabgrmod \tuT(\Lambda)$
for a finite dimensional algebra $\Lambda$. 
He also showed that this functor $\cH$ gives an equivalence precisely when $\gldim \Lambda < \infty$. 
 Thus if $\cH$ gives an equivalence, then it provides a canonical tilting object $\cH (\Lambda)$ of $\grmod \tuT(\Lambda)$.



In this paper we generalize  the Happel functor $\cH$ in the case where 
$\tuT(\Lambda)$  is replaced with  a finitely graded IG-algebra $A$. 
We study when this functor is fully faithful or is an equivalence. 
For this purpose we introduce the notion of homologically well-graded (hwg) IG-algebra, 
which can be characterized as an  algebra posses a homological symmetry which, a posteriori, guarantee that the algebra is IG.  
We prove that hwg IG-algebras is precisely the class of finitely  graded IG-algebras that the Happel functor is fully faithful. 
We also identify the class that the Happel functor gives an equivalence. 
As a consequence of our result, 
we see  that  if $\cH$ gives an equivalence, then it provides a canonical tilting object $\cH (T)$ of $\stabgrCM A$. 
For some special classes of finitely graded IG-algebras, 
our tilting objects $\cH(T)$ coincide with tilting object constructed in previous works. 
\end{abstract}

\tableofcontents

\newpage

\section{Introduction}\label{Introduction}

A central theme in the representation theory of
Iwanaga-Gorenstein (IG) algebra is the study of the stable category of Cohen-Macaulay (CM) modules. 
It was initiated by Auslander-Reiten \cite{AR}, Happel \cite{Happel} 
 and Buchweitz \cite{Buchweitz}, and has been studied by many researchers. 
%
The stable category of CM-modules $\stabCM A$ has a canonical structure of triangulated category. 
It is equivalent to the  singular derived category $\Sing A$ and is a triangulated category that is also important in 
 algebraic geometry and mathematical physics.  
The situation is the same with graded IG-algebras and the stable category $\stabgrCM A$ of graded CM-modules. 
Recently, tilting theory and cluster tilting theory  of the stable categories $\stabCM A$ and 
$\stabgrCM A$ 
are extensively studied and has many interaction with other areas 
(see the excellent survey \cite{Iyama: ICM}).

As will be soon recalled, the original Happel functor $\cH$ connects two important triangulated categories: 
the derived category $\sfD^{\mrb}(\mod \Lambda)$ and the stable category $\stabgrmod \tuT(\Lambda)$, 
so  served as a powerful tool to study these categories.
%
From tilting theoretic point of  view,  
the functor $\cH$ provides the existence of a canonical tilting object $\cH(\Lambda)$ in the stable category 
$\stabgrmod \tuT(\Lambda)$.

In this paper and \cite{higehaji}, we generalize the Happel functor $\cH$ by replacing the self-injective algebra $\tuT(\Lambda)$ 
 with a finitely graded IG-algebra $A=\bigoplus_{i = 0}^{\ell} A_{i}$. 
 In this paper, one of our main concern is the question when the canonical object $\cH(T)$ (which is $\cH(\Lambda)$ in the simplest case) 
  is a tilting object of $\stabgrCM A$. 
  For this purpose we introduce the notion of \emph{homologically well-graded} (hwg) algebras. 
 Our main concern in the paper is  a finitely graded  hwg IG-algebra $A= \bigoplus_{i = 0}^{\ell} A_{i}$, that is a finitely graded algebra 
 which is IG as well as hwg. 
 We provide several characterization of a finitely graded hwg IG-algebra.
  One of our main results characterizes  
a  hwg IG-algebra $A$ as a   finitely graded IG-algebra $A$  such that  $\cH(T)$ is a  tilting object, 
  whose  endomorphism algebra is  the Beilinson algebra $\nabla A$ - an algebra canonically constructed from $A$ (see \eqref{Beilinson algebra}).  
Another  result characterizes a finitely graded hwg IG-algebra 
as an  algebra posses a homological symmetry which, a posteriori, guarantee that the algebra is IG.  
This phenomena is looked as a generalization of the fact that 
 a Frobenius algebra is an algebra posses a symmetry which, guarantee that the algebra is self-injective. 
 Since a reason why  Frobenius algebras are of importance in  several areas is its  symmetry, 
 we can expect that hwg IG-algebra also play a basic role of other areas.

 In \cite{higehaji} we make use of the generalized Happel functor to study general aspect of  finitely graded IG-algebras and 
 their stable categories.    
For example, we show that the Grothendieck group $K_{0}(\stabgrCM A)$ is free of finite rank. 
We  expect that the generalized Happel functor can become an indispensable tool to study finitely graded IG-algebras.

\subsection{Results}

Now we explain the main results  and the notations used throughout the paper.

First we  recall, the original Happel functor $\cH$. 
Let  $\Lambda$ be a finite dimensional algebra  over some field $\kk$ 
and $\tuT(\Lambda) := \Lambda \oplus \tuD(\Lambda)$ the trivial extension algebra 
of $\Lambda$ by the bimodule  $\tuD(\Lambda) = \Hom_{\kk}(\Lambda, \kk)$,  
equipped with the grading $\deg \Lambda = 0, \deg \tuD(\Lambda) = 1$. 
In his pioneering work, Happel \cite{Happel book, Happel-1991} constructed a fully faithful triangulated functor 
\[ 
\cH : \sfD^{\mrb}(\mod \Lambda) \hookrightarrow \stabgrmod \tuT(\Lambda) 
\] 
and showed that it gives an equivalence if and only if $\gldim \Lambda < \infty$. 
Thus if $\cH$ gives an equivalence, then it provides a canonical tilting object $\cH (\Lambda)$ of $\grmod \tuT(\Lambda)$.

Although it looks like that the Happel functor $\cH$ is determined from $\Lambda$, 
there is a way to construct $\cH$ starting from $\tuT(\Lambda)$. 
In Section \ref{the Happel functor}, 
we generalize the Happel functor $\cH$ 
to the case where $\tuT(\Lambda)$ is replaced by a finitely graded IG-algebra  $A= \bigoplus_{ i = 0}^{\ell} A_{i}$. 
The generalized Happel functor $\cH$ has $\stabgrCM A$ as its codomain. 
The domain is the derived category $\sfD^{\mrb}(\mod^{[0, \ell -1]}A)$ of the abelian category  $\mod^{[0,\ell -1]}A \subset \grmod A$ which is the full subcategory consisting of $M = \bigoplus_{i \in \ZZ} M_{i}$ such that $M_{i} = 0$ for $i \notin [0, \ell -1]$. 
\begin{equation}\label{Intoroduction the Happel functor 3}
\cH : \sfD^{\mrb}(\mod^{[0,\ell -1]}A) \to \stabgrCM A. 
\end{equation}
The first fundamental  question about $\cH$ is  the following. 

\begin{question}
When is it fully faithful or an equivalence?
\end{question}

We focus on  the special case where $A$ is a graded self-injective algebra. 
Recall that 
graded  Frobenius algebras is a special class of graded self-injective algebras (for the definition, see Example \ref{Example graded Frobenius}).  
We can deduce an answer to the question from previous works
 by Chen \cite{Chen trivial}, Mori with the first author \cite{MM} and the second author \cite{Yamaura}. 
Namely,  the functor $\cH$ is fully faithful if and only if $A$ is  graded Frobenius.  
Moreover, if this is the case,  $\cH$ is  an equivalence if and only if $\gldim A_{0} < \infty$.

To state the second question, we need to introduce a graded $A$-module  $T$, 
which has been observed to play   an important role in the study of the Happel functor.  
\begin{equation}\label{canonical construction}
T :=  \bigoplus_{ i= 0}^{\ell-1} A(-i)_{\leq \ell-1} \in \mod^{[0, \ell -1]} A.
\end{equation}
The endomorphism algebra $\nabla A := \End (T)$ is called the \emph{Beilinson} algebra. 
We may identify it with the upper triangular matrix algebra below  via canonical isomorphism. 
\begin{equation}\label{Beilinson algebra} 
\nabla A:= \End_{\grmod A} (T) \cong 
\begin{pmatrix} 
A_{0} & A_{1} & \cdots & A_{\ell -1} \\
0         & A_{0} & \cdots  & A_{\ell -2} \\
\vdots & \vdots     &        & \vdots \\
0         & 0   & \cdots  & A_{0}
\end{pmatrix}.
\end{equation}
We denote by $\gamma$ the algebra homomorphism induced  by  the Happel functor $\cH$. 
\begin{equation}\label{Introduction morphism gamma}
\gamma: \nabla A = \End_{\grmod A}(T) \xrightarrow{ \cH_{T,T}}  \End_{\stabgrCM A} \cH(T).
\end{equation}

We note that $T$ is a  progenerator of $\mod^{[0, \ell-1]}A$. 
Moreover, by Morita theory, the functor $\sfq: = \Hom_{\mod^{[0, \ell -1]} A }(T, -)$ gives an equivalence  
\begin{equation}\label{Introduction nabla}
\mathsf{q}: \mod^{[0, \ell -1]} A \cong \mod \nabla A \textup{ such that } \sfq(T) = \nabla A.  
\end{equation}
Thus, we may regard the Happel functor $\cH$ as an exact functor from $\sfD^{\mrb}(\mod \nabla A) $ to $\stabgrCM A$. 
\[
\cH : \sfD^{\mrb}(\mod \nabla A) \xrightarrow{ \sfq^{-1} \ \cong \ } \sfD^{\mrb}(\mod^{[0, \ell-1]}A) \to \stabgrCM A.  
\]

The image $\cH(T)$  have been studied by many researchers. 
In the case where $A$ is a graded self-injective algebra, 
it is shown in \cite{Chen trivial, MM, Yamaura} that $\cH(T)$  is a tilting object of $\stabgrmod A$ if and only if $\gldim A_{0} < \infty$. 
Moreover, the morphism $\gamma$ is an isomorphism if and only if $A$ is graded Frobenius.  

%

 
As for graded IG-algebras $A$,  
it has been  shown that the object $\cH(T) \in \stabgrCM A$ is a tilting object 
or relates to  a construction of a tilting object of $\stabgrCM A$ 
in many other cases \cite{Kimura-1, Lu, LZ, MU}.  
However we would like to mention  that for a graded IG-algebra $A$, 
the graded module $\cH(T)$ does not give a tilting object of $\stabgrCM A$ in general (see for example \cite[Example 3.7]{LZ}). 
Thus our second question naturally arises. 

\begin{question}
When is  $\cH(T)$  a tilting object of $\stabgrCM A$ which satisfies the condition that  the map $\gamma$ is an isomorphism?
\end{question}

The answers of above two questions are given by 
the notion of  \textit{homologically well-graded (hwg)} algebras. 
The prototypical example of hwg algebras is the trivial extension algebra $\tuT(\Lambda) = \Lambda \oplus \tuD(\Lambda)$ of a finite dimensional algebra $\Lambda$. 
In the paper \cite{Chen trivial} mentioned above, 
Chen introduced the notion of well-gradedness for finitely graded algebra 
and 
showed that a well-graded self-injective algebra $A$ is graded Morita equivalent to $\tuT(\nabla A)$. 
Thus graded representation theory of $A$ is equivalent to that of $\tuT(\nabla A)$ and in particular 
Happel's  results can be applied. 
However for a finitely graded IG-algebra $A$ which is not self-injective, 
 well-gradedness is not enough to control CM-representation theory.
 We observed that  a key  to establish the Happel embedding is the following equation 
 \[
 \Hom_{\grmod \tuT(\Lambda)}(\Lambda, \tuT(\Lambda)(i)) = 0 \textup{ for } i \neq 1. 
 \]
The relationship between the equation and  well-gradedness is explained  in the begining of Section \ref{subsubsection: hwg algebra}. 
The point is that the equation admits a natural homological generalization, which yields  the definition of a hwg algebra.

Our main results show that a hwg IG-algebra gives complete answers to the above two questions.

%
%
%
%
%
%
%

\begin{theorem}[{Theorem \ref{main theorem 2}, Theorem \ref{tilting theorem}}]\label{introduction: main theorem 2} 
Assume that $\kk$ is a commutative  Noetherian ring and $A= \bigoplus_{i = 0}^{\ell} A_{i}$ is an 
IG-algebra that is finitely generated as a $\kk$-module. 
Then the following conditions are equivalent.
\begin{enumerate}[(1)]
\item 
$A$ is hwg  
(resp. 
$A$ is hwg  and $A_{0}$ satisfies the condition (F)).

\item 
The Happel  functor $\cH$ is fully faithful 
(resp. equivalence). 

\item 
The morphism $\gamma$ is an isomorphism 
and $\Hom_{\stabgrCM A}(\cH(T), \cH(T)[n]) = 0$ for $n \neq 0$ 
(resp.   $\gamma$ is an isomorphism and $\cH(T)$ is a tilting object of $\stabgrCM A$). 

\end{enumerate}
\end{theorem}
The condition (F) is defined in Definition \ref{definition: finite local dimension}. 
It is  a condition on finiteness of homological dimensions on $A_{0}$ which is weaker than the condition $\gldim A_{0} < \infty$. 
But  this condition  is equivalent to $\gldim A_{0}< \infty$ in the case where $\kk$ is a complete local ring and hence in particular is a field. 

In the case where $\kk$ is a field, 
as we mentioned above a typical example of hwg algebra is $\tuT(\Lambda) $ for some finite dimensional algebra $\Lambda$. 
We can apply  the equivalence(s) (1) $\Leftrightarrow$ (2) to it and recover Happel's original result. 
However Happel's  proof of the implication $(1) \Rightarrow (2)$ of respective cases made use of  the fact that the stable category $\stabgrmod \tuT(\Lambda)$ has Auslandr-Reiten triangles. 
Since we do not know that  Auslander-Reiten triangles may not make sense  in the case where $\kk$ is not a field, 
we can not use the Happel method and need to develop our method. 

Our method relies on the decompositions of complexes of graded injective or projective modules established in \cite{adasore}.  
As a by-product we are able to deal with  the case where $A$ is not necessary IG. 
In  Lemma \ref{seiri lemma 2} and Proposition \ref{seiri proposition 2}, 
we study the relation between existence of a generator in the singular derived category of $A$ and 
the finiteness of homological dimensions of $A_{0}$. 
To the  best of our knowledge, all previous results about such a relation  only in the case of  graded  IG-algebras $A$. 
Thus, although it is beyond the main theme of the paper,  these results are of their  own interest.

%

\bigskip

To finish the introduction, we explain other results of the paper.  
A graded algebra which is both hwg and IG has a nice structure.
We show that if a finitely graded algebra $A = \bigoplus_{i =0}^{\ell} A_{i}$ is hwg IG, 
then the subalgebra $A_{0}$ of degree $0$ elements is Noetherian and 
the highest degree submodule $A_{\ell}$ is a cotilting bimodule over $A_{0}$;  see Definition  \ref{definition of cotilting modules}. 

Recall that the key property of a cotilting bimdoule $C$  over a Noetherian algebra is that 
the duality $\RHom_{\Lambda}(-, C)$ gives a contravariant equivalence between the derived categories 
$\sfD^{\mrb}(\mod \Lambda)$ and $\sfD^{\mrb}(\mod \Lambda^{\op})$.

In Theorem \ref{main theorem} we give characterizations of hwg IG-algebras 
which are neither stated in terms of  the stable category nor the Happel functor. 
Among other things,  we verify that 
a hwg IG-algebra is precisely a finitely graded Noetherian  algebra $A = \bigoplus_{i =0}^{\ell}A_{i}$ which has the following properties: 
(i) the $A_{0}$-$A_{0}$-bimodule  $A_{\ell}$ is a cotilting bimodule. 
(ii) $A$ has  a homological symmetry given by the duality induced from the cotilting bimodule $A_{\ell}$.

In the case $\kk$ is a field and $\Lambda$ is a finite dimensional $\kk$-algebra, 
 the $\kk$-dual bimodule $\tuD(\Lambda)$ is an example of cotilting bimodule 
 and the duality induced by $\tuD(\Lambda)$ is nothing but the $\kk$-duality, 
 i.e., $\RHom_{\Lambda}(-, \tuD(\Lambda))  \cong \tuD(-)$. 
 Using this fact,  
 we observe  in Example \ref{Example graded Frobenius} that  a finite dimensional graded self-injective algebra $A$ is hwg if and only if it is
a graded Frobenius. 
 In this sense, 
a hwg IG-algebra can be looked  as a generalization of a graded Frobenius algebra obtained by replacing 
the bimodule $\tuD(\Lambda)$ with a general cotilting bimodule $C$.

It follows from a classical result by Fossum-Griffith-Reiten \cite[Theorem 4.32]{FGR} that 
when $\Lambda$ is a Noetherian algebra and $C$ is a cotilting bimodule, 
 the trivial extension algebra $A = \Lambda \oplus C$  is IG. 
We prove that if we equip $A$ with the grading $\deg \Lambda = 0, \deg C = 1$,  
then it become  hwg IG.  
Moreover, in Corollary \ref{ell = 1 corollary},  we show that a graded algebra $A = A_{0} \oplus A_{1}$ concentrated in degree $0,1$ 
is hwg IG if and only if it is obtained in such a way.

Now it is natural to recall the following result of commutative Gorenstein algebras 
due to Foxby \cite{Foxby} and Reiten \cite{Reiten}.  
Namely, 
the trivial extension algebra 
$A = \Lambda \oplus C$ of a commutative Noetherian local algebra $\Lambda$ by a (bi)module $C$ 
is IG if and only if $C$ is a cotilting (bi)module. 
Thus with our terminology this theorem says that,  in commutative  local setting, 
every graded IG-algebra $A = A_{0} \oplus A_{1}$ concentrated in degree $0,1$ is hwg IG.
We prove the same result is true for a commutative finitely graded IG-algebra. 

\begin{theorem}[{Theorem \ref{commutative Gorenstein theorem}}]
A commutative local finitely graded IG-algebra 
$A = \bigoplus_{i= 0}^{\ell} A_{i}$ is hwg. 
\end{theorem}

\subsection{Organization of the paper}
The paper is organized as follows. 
In Section \ref{graded modules}, first we fix notations for graded modules and their derived categories. 
Then we recall a decomposition of a complex $I \in \sfC(\GrInj A)$ of graded injective modules introduced in \cite{adasore}.  
In Section \ref{the Happel functor} we give the construction of the Happel functor and recall related results.  
In Section \ref{Homologically well-graded algebras} we introduce a notion of homologically well-graded (hwg) algebras. 
In Section \ref{characterization} we give characterizations of hwg algebras and show that it can be looked as a generalization of graded Frobenius algebras.   
In Section \ref{the Happel functor and a homologically well-graded IG-algebra}, 
we give characterizations of fully faithfulness of $\cH$ (Theorem \ref{main theorem 2}) 
and characterizations of when $\cH$ gives an equivalence (Theorem \ref{tilting theorem}). 
In Section \ref{Examples and constructions}, we give several examples and constructions of hwg IG-algebras.  
We observe that being hwg IG is more robust than being IG. 
For example, even though taking Veronese algebras and Segre products do not preserve IG-algebras,  
these operations preserve hwg IG-algebras. 
In Section \ref{commutative case}, 
we focus on the  commutative case and generalize a result of  Fossum-Griffith-Reiten, Foxby and Reiten  \cite{FGR, Foxby, Reiten}. 
 In Section \ref{Graded derived Frobenius extensions} 
 we discuss the definition of hwg IG-algebras.

\vspace{10pt}
\noindent
\textbf{Acknowledgment}
The authors thank anonymous referee for his/her careful reading and numerous comments 
about mathematical contents and readability. 
The first author  was partially  supported by JSPS KAKENHI Grant Number 26610009.
The second author  was partially  supported by JSPS KAKENHI Grant Number 26800007.

\subsection{Notation and convention}

\subsubsection{Algebras, modules and bimodules}

Throughout the paper $\kk$ denotes a commutative ring. 
An algebra $\Lambda$ is always a $\kk$-algebra. 
Unless otherwise stated, the word  ``$\Lambda$-modules" means right $\Lambda$-modules.  
We denote by $\Mod \Lambda$ the category of $\Lambda$-modules. 
We denote by $\Proj \Lambda$ (resp. $\Inj \Lambda$) 
the full subcategory of projective (resp. injective) $\Lambda$-modules. 
We denote  by $\proj  \Lambda$ the full subcategory of finitely generated projective $\Lambda$-modules. 

We set $\Hom_{\Lambda} := \Hom_{\Mod \Lambda}$. 
Note that $\Hom_{\Lambda}$ also denotes the Hom-space of the derived category $\sfD(\Mod \Lambda)$.

We denote the opposite algebra  by $\Lambda^{\op}$. 
We identify left $\Lambda$-modules with (right) $\Lambda^{\op}$-modules.   
A $\Lambda$-$\Lambda$-bimodule $D$ is always assumed to be $\kk$-central, 
i.e., $ad = da $ for $d \in D, \ a \in \kk$. 
For a $\Lambda$-$\Lambda$-bimodule $D$, 
we denote by $D_{\Lambda}$ and ${}_{\Lambda} D$ 
the underlying right  and left $\Lambda$-modules respectively.

\subsubsection{Categories of cochain complexes, the homotopy categories and the derived categories}

For an additive category $\cA$, we denote by $\sfC(\cA)$ and $\sfK(\cA)$  
the category of cochain complexes and cochain morphisms 
and its homotopy category respectively. 
For complexes $X, Y \in \sfC(\cA)$, 
we denote by $\Hom_{\cA}^{\bullet}(X, Y)$ the $\Hom$-complex. 
For an abelian category $\cA$, we denote by $\sfD(\cA)$ the derived category of $\cA$.

We denote the derived functor of $\Hom_{\cA}$ by $\RHom_{\cA}$.

For an algebra $\Lambda$, we set $\Hom_{\Lambda} := \Hom_{\sfD(\Mod \Lambda)}$ and  $\RHom_{\Lambda} := \RHom_{\Mod \Lambda}$.

\subsubsection{Triangulated categories}\label{tri cat} 

A triangulated category $\sfT$ is always assumed to be linear over the base commutative ring $\kk$. 
Let $\mathsf{U} , \mathsf{V} \subset \sfT $ be full triangulated subcategories. 
We denote by $\sfU * \sfV \subset \sfT$ to be the full subcategory 
consisting of those objects $X$ which fit into an exact triangle 
$U \to X \to V \to $ with $U \in \sfU, V \in \sfV$. 
If $\Hom_{\sfT}(U, V) = 0$  for all  $U \in \sfU$ and all $V \in \sfV$, we write $U \perp V$. 

Let $X \in \sfT$ be an object. 
We denote by $\thick X$ the thick closure of $X$, that is, the smallest triangulated subcategory of $\sfT$ containing $X$ that is closed under direct summands. 
In other words, it is a triangulated subcategory of $\sfT$ 
consisting of objects which are constructed from $X$ by taking shifts, cones and direct summands. 
An object $X \in \sfT$ is said to be  a  \emph{tilting object} of $\sfT$ if $\thick X = \sfT$ and $\Hom_{\sfT}(X ,X [n] ) = 0$ for $n \neq 0$.

\section{Graded modules and their derived categories}\label{graded modules}

In this  paper,   a graded algebra  $A = \bigoplus_{i \geq 0} A_{i}$ is always finitely graded,  
that is $A_{i} = 0$ for $i \gg 0$. 
Moreover, we always assume that the maximal degree  
$\ell := \max \{ i \in \NN  \mid A_{i} \neq 0 \}$ of $A$ 
is positive, i.e., $\ell \geq 1$. 

In this Section \ref{graded modules}, we set notations related to graded modules and their derived categories, 
and collect basic facts which is used in the later sections. 

\subsection{Graded algebras and graded modules}
We fix notations for graded modules and recall basic facts. 
 For  details, we refer the readers to \cite{NV:Graded and Filtered}.

We denote by $\GrMod A$ the category of graded (right) $A$-modules\footnote{See Remark \ref{202003100849} for the expression of a graded module $M$} $M= \bigoplus_{i \in \ZZ} M_{i}$ 
and graded $A$-module degree-preserving homomorphisms $f: M \to N$
i.e., $f(M_{i}) \subset N_{i}$ for $i \in \ZZ$.

For a graded $A$-module $M$ and an integer $j \in \ZZ$, 
we define the shift $M(j) \in \GrMod A$ by $(M(j))_{i} = M_{i+j}$.  
The truncation $M_{\geq  j}$  is a graded submodule of $M$ defined by 
\[
(M_{\geq j})_{i} :=
\begin{cases} 
M_{i} &  (i \geq j), \\  0  & ( i < j). 
\end{cases}
\]
We set $M_{ < j} := M/M_{\geq j}$ so that we have an exact sequence $ 0\to M_{\geq j} \to M \to M_{< j} \to 0$.

For $M, N \in \GrMod A, \ n \in \NN$ and $i \in \ZZ$,  
we set $\grExt_{A}^{n}(M,N)_{i}:= \Ext_{\GrMod A}^{n}(M,N(i))$ 
and 
\[
\grExt_{A}^{n}(M, N) := \bigoplus_{i \in \ZZ} \grExt_{A}^{n}(M, N)_{i} = \bigoplus_{i \in \ZZ} \Ext_{\GrMod A}^{n}(M, N(i)). 
\]
We regard $\grExt_{A}(M,N)$ as a graded $\kk$-module with the grading given as in the formula. 

We set $\grHom_{A}(M,N) := \grExt_{A}^{0}(M, N)$. 
We note the obvious equations \[
\grHom_{A}(M,N) = \bigoplus_{i\in \ZZ}\Hom_{\GrMod A}(M, N(i)), \ \ 
\grHom_{A}(M,N)_{0} = \Hom_{\GrMod A}(M,N).\] 
%

We may regard $A$ and $\Lambda$ as  graded $A$-$A$-bimodules. 
Then the canonical projection $p: A \to \Lambda$ is a homomorphism of graded $A$-$A$-bimodules. 
We may  identify  $\grHom_{A}(A,M)$ with $M$ via a canonical map $\grHom_{A}(A,M) \to M, \  f\mapsto f(1)$. 
Moreover 
we may identify $\grHom_{A}(\Lambda, M)$ with the graded submodule $\{ m \in M \mid am = 0 \ (\forall a \in A_{\geq 1}) \}$ of $M$ 
via the induced injective map $\grHom(p, M): \grHom_{A}(\Lambda, M) \hookrightarrow \grHom_{A}(A,M) \cong M$. 
It is shown in \cite[Lemma 2.8]{adasore} that $\grHom_{A}(\Lambda, M)$ is an essential graded submodule of $M$. 
We leave the verification of the following lemma to the readers. 

\begin{lemma}
Assume that $\kk$ is a field and $A$ is finite dimensional over $\kk$. 
Let $M$ be a finite dimensional graded $A$-module. 
Then for an integer $i$, we have $\grHom_{A}(\Lambda, M)_{i} \neq 0$ if and only if $( \soc M)_{i} \neq 0$. 
\end{lemma}

\subsubsection{The subcategory $\Mod^{I} A$ of $\GrMod A$}
For a subset $I \subset \ZZ$, we denote by $\Mod^{I} A \subset \GrMod A$  the full subcategory consisting of graded $A$-modules $M$ such that $M_{i} = 0$ for $i \in \ZZ \setminus I$. 
We note that $\Mod^{I} A$ is an abelian subcategory of $\GrMod A$.

Let $i$ be an integer. 
For notational simplicity, we set $\Mod^{\leq i} A := \Mod^{( - \infty, i]} A$ and $\Mod^{\geq i} A := \Mod^{[i, \infty) } A$. 
We may regard the assignment  $M \mapsto M_{\geq i}$ as a functor $\GrMod A \to \Mod^{\geq i } A$. 
It is a right adjoint functor of the embedding functor $\mathsf{em}_{\geq i}: \Mod^{\geq i}A \hookrightarrow \GrMod A$.  
\begin{equation}\label{202003091912}
\mathsf{em}_{\geq i}: \Mod^{\geq i } A \rightleftarrows \GrMod A: (-)_{\geq i}
\end{equation}
Note that we have $M = (\mathsf{em}_{\geq i}(M))_{\geq i}$. 
The functor $(-)_{\leq i}:  \GrMod A \to \Mod^{\leq i} A , \ M \mapsto M_{\leq i}$ 
is a left adjoint functor of the embedding 
$\mathsf{em}_{\leq i} : \Mod^{\leq i} A \to \GrMod A$. 
\begin{equation}\label{202003091913}
(-)_{\leq i}:  \GrMod  A \rightleftarrows \Mod^{\leq i} A: \mathsf{em}_{\leq i}.
\end{equation}
Note that we have $M = (\mathsf{em}_{\leq i}(M))_{\leq i}$.

\subsubsection{A canonical embedding $\Mod \Lambda \hookrightarrow \GrMod A$.}

For notational simplicity we always set $\Lambda := A_{0}$. 
%
%
We  regard a $\Lambda$-module $N$ as a graded $A$-modules concentrated in degree $0$.  
The category $\Mod \Lambda$ of $\Lambda$-modules is identified with  the full subcategory $\Mod^{0} A= \Mod^{[0,0]} A$.

\begin{remark}\label{202003100849} 
Let $M$ be a graded $A$-module and $i$ an integer.  Then we regard the $i$-degree part $M_{i}$ as an  ungraded $\Lambda$-module. 
We remark that by the above convention, the $i$-degree part $M_{i}$ is regarded as a graded $\Lambda$-module concentrated in degree $0$. 
Therefore, the underlying graded $\Lambda$-module of $M$ is $\bigoplus_{i \in \ZZ} M_{i}(-i)$.

We remark that $M_{\geq i}$ is a subobject of $M$ and $M_{\leq i}$ is a quotient object of $M$ in $\GrMod A$.    
Therefore for example, we have $(M_{\geq i})_{\leq i} \neq  M_{i}$, but $(M_{\geq i})_{\leq i} = M_{i}(-i)$ in $\GrMod A$. 
\end{remark}

\subsubsection{The functor $\grHom_{\Lambda}(A, -): \Mod \Lambda \to \GrMod A$}

We introduce a functor $\grHom_{\Lambda}(A, -): \Mod \Lambda \to \GrMod A$ which plays a key role in the paper.

For this purpose, it is convenient to work with the category $\GrMod \Lambda$ of graded $\Lambda$-modules 
where we regard $\Lambda$ as a graded algebra concentrated in degree $0$. 
Let $\sfU: \GrMod A \to \GrMod \Lambda$ be the functor which sends a graded $A$-module $M$ to its underlying graded $\Lambda$-module $\bigoplus_{i \in \ZZ} M_{i}( -i)$.  
Observe that 
the functor $\sfU$ is obtained as the tensor product $- \otimes_{A} A$ where we regard $A$ as a graded $A$-$\Lambda$-bimodule. 
Therefore we have the following adjoint pair 
\begin{equation}\label{202003091829}
\sfU = -\otimes_{A} A : \GrMod A \rightleftarrows \GrMod \Lambda : \grHom_{\Lambda}(A, -). 
\end{equation}

Since $\Lambda$ is concentrated in degree $0$, we have 
\begin{equation}\label{202003102001}
\grHom_{\Lambda}(M,N)_{i} = \prod_{j \in \ZZ} \Hom_{\Lambda}(M_{j -i}, N_{j})
\end{equation} 
for $M, N \in \GrMod \Lambda$. 
Thus in particular we have 
\begin{equation}\label{202003091839}
\grHom_{\Lambda}(A, N)_{i} = \bigoplus_{i \leq j \leq \ell + i}\Hom_{\Lambda}(A_{j-i}, N_{j}).
\end{equation}

\begin{definition}\label{202003141933} 
We define a functor $\grHom_{\Lambda}(A, -): \Mod \Lambda \to \GrMod A$ to be the following composition 
\[
\grHom_{\Lambda}(A, -): \Mod \Lambda \hookrightarrow \GrMod \Lambda \xrightarrow{ \ \grHom_{\Lambda}(A, -) \ } \GrMod A
\]
where the first arrow is a canonical embedding that regards $\Lambda$-modules as graded $\Lambda$-modules 
concentrated in degree $0$. 
\end{definition}

\begin{remark}
The symbol $\grHom_{\Lambda}(A, -)$ only denotes the functor $\grHom_{\Lambda}(A, -) : \Mod \Lambda \to \GrMod A$ 
and does not denote  the functor $\grHom_{\Lambda}(A, -) : \GrMod \Lambda \to \GrMod A$ in the sequel. 
\end{remark}
It follows from \eqref{202003091839} that, for a $\Lambda$-module $N$, we have $\grHom_{\Lambda}(A, N) \in \Mod^{[- \ell, 0]} A$.

\begin{lemma}\label{202003101955} 
Let $M \in \GrMod A, \ N \in \Mod \Lambda$ and $k \in \ZZ$. 
Then we have  an isomorphism  
\[
\grHom_{A}(M_{\leq k}, \grHom_{\Lambda}(A, N)) \cong \grHom_{A}(M, \grHom_{\Lambda}(A, N))_{\geq- k}.
\]
in $\GrMod \kk$.
\end{lemma}

\begin{proof}
Let $i$ be an integer. We have the following isomorphisms 
\[
\begin{split}
\grHom_{A}(M, \grHom_{\Lambda}(A, N))_{i} \cong  
\grHom_{\Lambda}(\sfU(M), N)_{i} \cong \Hom_{\Lambda}(M_{ -i}, N)  
\end{split}
\]
where 
the first isomorphism is deduced from the adjoint pair \eqref{202003091829} 
and the second isomorphism is a special case of \eqref{202003102001}.
In the same way, we obtain 
\[
\grHom_{A}(M_{\leq k}, \grHom_{\Lambda}(A, N))_{i} 
= 
\begin{cases}
\Hom_{\Lambda}(M_{-i}, N) &  ( i \geq -k) \\
0 & ( i < k).  
\end{cases}
\]
Therefore we conclude the desired isomorphism. 
\end{proof}

\subsection{The derived category of graded modules}

 For complexes $M, N \in \sfC(\GrMod A)$ of graded $A$-modules, 
we denote by $\grHom_{A}^{\bullet}(X, Y)$ the graded $\Hom$-complex. 
Namely, for $i \in \ZZ$ we set $\cpxgrHom_{A}(M, N)_{i} := \cpxHom_{\GrMod A}(M, N(i))$ 
and 
\[
\cpxgrHom_{A}(M,N) := \bigoplus_{i \in \ZZ}\cpxgrHom_{A}(M, N)_{i} = \bigoplus_{ i \in\ZZ}\cpxHom_{\GrMod A}(M, N( i)).  
\]
We regard $\cpxgrHom_{A}(M,N)$ as an object of $\sfC(\GrMod \kk)$ with the grading given as in the formula.

For objects $M, N \in \sfD(\GrMod A), \ n \in \NN$ and $i \in \ZZ$,  
we set $\grRHom_{A}(M,N)_{i}:= \RHom_{\GrMod A}(M,N(i))$ 
and 
\[
\grRHom_{A}(M, N) := \bigoplus_{i \in \ZZ} \grRHom_{A}(M, N)_{i} = \bigoplus_{i \in \ZZ} \RHom_{\GrMod A}(M, N(i)). 
\]
We regard $\grRHom_{A}(M,N)$ as an object of $\sfD(\GrMod \kk)$ with the grading given as in the formula. 
We note that for  $M, N \in \GrMod A$ and $ n \in \NN$,  we have a natural isomorphism 
\[
\tuH^{n}(\grRHom_{A}(M,N)) = \grExt_{A}^{n}(M, N).
\]

\subsubsection{The embedding $\mathsf{em}_{[i,j]}: \sfD(\Mod^{[i,j]} A) \to \sfD(\GrMod A)$}

First note that the functors in the adjoint pairs \eqref{202003091912} and \eqref{202003091913} are exact. 
Therefore, we obtain the following adjoint pairs of derived categories. 
\[
\mathsf{em}_{\geq i}: \sfD( \Mod^{\geq i } A ) \rightleftarrows \sfD(\GrMod A): (-)_{\geq i}, \ \ 
(-)_{\leq i}:  \sfD(\GrMod  A) \rightleftarrows \sfD(\Mod^{\leq i} A): \mathsf{em}_{\leq i}.
\]
Note that we have $M= (\mathsf{em}_{\geq i}(M))_{\geq i}$  for $M \in \sfD(\Mod^{\geq i} A)$. 
Therefore, the functor $\mathsf{em}_{\geq i}$  is fully faithful. 
Similarly, we have  $M = (\mathsf{em}_{\leq i}(M))_{\leq i}$  for $M \in \sfD(\Mod^{\leq i} A)$.
Therefore, the functor  $\mathsf{em}_{\leq i}$ is fully faithful.

Let $i, j$ be integers such that $i \leq j$. 
We denote by  $\mathsf{em}_{[i,j]}: \Mod^{[i,j]}A \to \GrMod A$ the embedding functor. 
We denote the induced functor $\mathsf{em}_{[i,j]} :\sfD(\Mod^{[i,j]} A) \to \sfD(\GrMod A)$ by the same symbol.

\begin{lemma}\label{embedding lemma}
The functor $\mathsf{em}_{[i,j]} :\sfD(\Mod^{[i,j]} A) \to \sfD(\GrMod A)$ is fully faithful. 
\end{lemma}

\begin{proof}
The embedding $\mathsf{em}_{[i,j]}: \Mod^{[i,j]} A \to \GrMod A$ is the composition of the embedding functors 
\[
\mathsf{em}_{[i,j]} : \Mod^{[i,j]} A \xrightarrow{ \mathsf{em}'_{\leq j}} \Mod^{\geq i} A \xrightarrow{ \ \mathsf{em}_{\geq i} \ } \GrMod A. 
\]
Therefore the functor $\mathsf{em}_{[i,j]} :\sfD(\Mod^{[i,j]} A) \to \sfD(\GrMod A)$ is the 
following composition 
\[
\mathsf{em}_{[i,j]} : \sfD(\Mod^{[i,j]} A) \xrightarrow{ \mathsf{em}'_{\leq j}} \sfD(\Mod^{\geq i} A) 
\xrightarrow{ \ \mathsf{em}_{\geq i} \ }  \sfD(\GrMod A). 
\]

We already shown that the induced functor $\mathsf{em}_{\geq i} : \sfD(\Mod^{\geq i} A) \to \sfD(\GrMod A)$ is 
fully faithful. 
We can prove that the induced functor $\mathsf{em}'_{\leq j} : \sfD(\Mod^{[i,j]} A) \to \sfD(\Mod^{\geq i}  A)$ is 
fully faithful  
by a similar argument that proves the functor 
$\mathsf{em}_{\leq j} : \sfD(\Mod^{ \leq  j} A) \to \sfD(\GrMod A)$ is fully faithful. 
Thus we conclude that the functor $\mathsf{em}_{[i,j]} :\sfD(\Mod^{[i,j]} A) \to \sfD(\GrMod A)$ is fully faithful. 
\end{proof}

\begin{remark}\label{202003141726}
Let $i$ be an integer. 
By Remark \ref{202003100849} taking the $i$-th degree part yields an exact functor $(-)_{i} : \GrMod A \to \Mod \Lambda$. 
We denote by the same symbol  $(-)_{i} : \sfD(\GrMod A) \to \sfD(\Mod \Lambda)$  the induced functor. 

We may identify the category $\Mod \Lambda$ with the full subcategory $\Mod^{0} A := \Mod^{[0,0]} A$.
On the other hand, 
by Lemma \ref{embedding lemma}, $\sfD(\Mod^{0} A)$ is regarded as a full subcategory of $\sfD(\GrMod A)$. 
Therefore, we may identify $\sfD(\Mod \Lambda)$ with a full subcategory of $\sfD(\GrMod A)$ consisting of those objects $M$ such that 
$M_{i} = 0$ in $\sfD(\Mod \Lambda)$ for $i \neq 0$. 

We remark that,  for example, we have   $(M_{\geq i})_{\leq i} \neq M_{i}$, but $(M_{\geq i})_{\leq i} = M_{i}(-i)$ in $\sfD(\GrMod A)$. 
\end{remark}

\subsubsection{The functor $\grRHom_{\Lambda}(A, -): \sfD(\Mod \Lambda) \to \sfD(\GrMod A)$}\label{202003151619}

We denote by $\grRHom_{\Lambda}(A, -): \sfD(\Mod \Lambda) \to \sfD(\GrMod A)$ 
the derived functor of the functor $\grHom_{\Lambda}(A, -): \Mod \Lambda \to \GrMod A$ defined in Definition \ref{202003141933}.

We note an isomorphism $\grRHom_{\Lambda}(A, M)_{i} \cong \RHom_{\Lambda}(A_{ -i}, M)$ in $\sfD(\Mod \Lambda)$. 
Therefore $\grRHom_{\Lambda}(A, M)$ belongs to $\sfD(\Mod^{[-\ell, 0]}A)$. 
Since $A$ is finitely graded, the object $\grRHom_{\Lambda}(A, M) \in \sfD(\GrMod A)$ is of bounded cohomology 
if and only if so are $\RHom_{\Lambda}(A_{i}, M) \in \sfD(\Mod \Lambda)$ for $i = 0, 1, \cdots, \ell$.

\subsection{Decomposition of a  complex of graded injective $A$-modules}
We recall from \cite{adasore} a decomposition of  a complex $I$ of  graded injective $A$-modules. 

\subsubsection{Decomposition of a   graded injective $A$-module}

First we deal with a decomposition of a graded injective $A$-module $I$.

We denote by  $\GrInj A$ the full subcategory of  graded injective modules. 
For an integer $i \in \ZZ$, 
we denote  by
$\frki_{i}: \GrInj A \to \Inj \Lambda$ the functor $\frki_{i}I := \grHom_{A}(\Lambda, I)_{i}$ 
and define a graded injective $A$-module  
$\frks_{i}I := \grHom_{\Lambda}(A,  \frki_{i}I)(-i)$.

Roughly speaking $\frki_{i}I$ is a set of cogenerators in degree $i$ and 
$\frks_{i} I $ is a maximal graded submodule of $I$ cogenerated in degree $i$. 
We collect basic properties of these functors from \cite{adasore}. 

\begin{lemma}[{\cite[Lemma 2.7, Corollary 2.10]{adasore}}]\label{202003111530}
For $I \in \GrInj A$,  the following assertions hold. 
\begin{enumerate}[(1)]
\item 
We have a canonical  isomorphism of graded $A$-modules 
\[
 I  \cong \bigoplus_{ i\in \ZZ} \frks_{i} I.
 \]
 
 \item For $M \in \GrMod A$, we have the following isomorphism of $\kk$-modules
 \[
\Hom_{\GrMod A}(M, I ) \cong \prod_{i\in \ZZ} \Hom_{\Lambda}(M_{i} , \frki_{i}I). 
\]
 \end{enumerate}
 \end{lemma}

 We give further properties of the functors $\frki_{i}$ and $\frks_{i}$. 
%
%
We note that for $j \in \ZZ$ we have $(\frks_{i}I)_{j} \cong \Hom_{\Lambda}(A_{i-j}, \frki_{i} I)$. 
In particular we have $\frks_{i} I \in \Mod^{[i -\ell, i]}A$.


\begin{lemma}\label{202003082113}
The following assertions hold. 
\begin{enumerate}[(1)] 

\item 
Let $I$ be a graded injective $A$-module and $i \in \ZZ$.   
 Then for $0 \leq k \leq \ell -1$, we have 
 \[
\grHom_{A}(A_{\leq k} , \frks_{i} I)  \cong  (\frks_{i} I)_{\geq i -k}.
\] 

\item 

Let $I, J\in \grInj A$ and $i, j \in \ZZ$. 
Assume that $ j < i$, then we have 
$\Hom_{\GrMod A}(\frks_{j} J, \frks_{i}I) = 0$. 

\end{enumerate}
\end{lemma}

 \begin{proof}
 (1) follows from Lemma \ref{202003101955}. 
 
 (2) 
 Applying (1) to the case $k =0$,  
we obtain the following equality for $i, j  \in \ZZ$  
\[
\frki_{j} \frks_{i} I = 
\begin{cases} 
\frki_{i} I & ( i = j), \\
0 & ( i \neq j). 
\end{cases}
\]
Combining this equality with Lemma \ref{202003111530} (2), 
we conclude $\Hom_{\GrMod A}(\frks_{j} J, \frks_{i}I) = 0$ as desired. 
\end{proof}

\begin{definition}
For  a subset $I \subset \ZZ$, we define a full subcategory $\Inj^{I\textup{-cog}} A \subset \GrMod A$ to be 
\[
\Inj^{I \textup{-cog}}A := \{ I \in \GrInj A\mid \frki_{j} I = 0 \ ( \forall j \in \ZZ \setminus I)\}.
\]
For an integer $i \in \ZZ$, we set 
$\Inj^{\leq i \textup{-cog}}A:= \Inj^{( -\infty, i] \textup{-cog}}A$ and 
$\Inj^{> i \textup{-cog}}A:= \Inj^{(i, \infty) \textup{-cog}}A$.
\end{definition}

Combining Lemma \ref{202003111530} (1) and Lemma \ref{202003082113} (2), 
we obtain the following corollary.

\begin{corollary}\label{202003111555}
Let $i \in \ZZ$. 
Then for $I \in \Inj^{\leq i\textup{-cog}}A$ and $J \in \Inj^{> i \textup{-cog}}A$, 
we have 
$\Hom_{\GrMod A}( I, J) =0$. 
\end{corollary}

For $I \in \GrInj A$ and $i \in \ZZ$, 
we set $\frks_{\leq i} I := \bigoplus_{j \leq i} \frks_{j}I$  and $\frks_{> i} I := \bigoplus_{j >i} \frks_{j}I$. 
Then we have $I \cong (\frks_{\leq i} I) \oplus (\frks_{> i} I)$ by Lemma \ref{202003111530} (1). 
It follows from Corollary \ref{202003111555} that
a morphism $f:I \to J$ in $\GrInj A$ is of the following form 
\[
f: I = (\frks_{\leq i} I) \oplus (\frks_{> i} I) \xrightarrow{\small \begin{pmatrix} \frks_{\leq i } (f) & * \\ 0 & \frks_{> i}(f) \end{pmatrix} } 
(\frks_{\leq i} J) \oplus (\frks_{> i} J) = J
\]

\subsubsection{Decomposition of a  complex of graded injective $A$-modules}

By abuse of notations, 
we denote the functors  
$\frki_{i}: \sfC(\GrInj A ) \to \sfC(\Inj \Lambda), \ 
\frks_{i}: \sfC(\GrInj A) \to \sfC(\GrInj A)$ 
induced  from the functors  $\frki_{i}: \GrInj A \to \Inj \Lambda, \ \frks_{i}: \GrInj A \to \GrInj A$. 
Namely, for $I \in \sfC(\GrInj A)$ we set $\frki_{i}I := \cpxgrHom_{A}(\Lambda, I)$ 
and $\frks_{i}I := \cpxgrHom_{\Lambda}(A, \frki_{i} I)( -i)$.

Let $I=( \bigoplus_{n\in \ZZ} I^{n},\{\partial^{n}_{I}\}_{n\ \in \ZZ} )$ be an object of $\sfC(\GrInj A)$ and $i \in \ZZ$. 
Then by Lemma \ref{202003111530} (1), the underlying cohomological graded object  of $I$ is  $\bigoplus_{ i \in \ZZ} \frks_{i}I$. 
Namely, the component $I^{n}$ of  the cohomological degree $n$ is $\bigoplus_{j \in \ZZ} \frks_{i}( I^{n})$. 
Note that the differential $\partial_{I}$ does not preserves $\frks_{i} I$. 
So $I$ dose not coincide with $\bigoplus_{ i \in \ZZ} \frks_{i}I$ as complexes.
For $n \in \ZZ$,  
the $n$-th differential $\partial_{I}^{n}: I^{n} \to I^{n +1}$ is of the following form 
\[
\partial_{I}^{n} : I^{n} = (\frks_{\leq i} I^{n}) \oplus (\frks_{> i} I^{n} ) 
\xrightarrow{\small \begin{pmatrix} \frks_{\leq i } (\partial_{I}^{n}) & * \\ 0 & \frks_{> i}(\partial_{I}^{n}) \end{pmatrix} } 
(\frks_{\leq i} I^{n +1} ) \oplus (\frks_{> i} I^{n +1}) = I^{n +1}.
\]
Therefore, 
we obtain a subcomplex  $ \frks_{\leq i}I := ( \bigoplus_{ n \in \ZZ} \frks_{\leq i} (I^{n}), \frks_{\leq i}(\partial_{I}))$ of 
 $I$.  
We set $\frks_{ > i} I := I /\frks_{\leq i} I$. 
We note that the underlying cohomological graded object of $\frks_{>i}I$  is $\bigoplus_{n \in \ZZ} \frks_{> i} (I^{n})$. 
These complexes  fit into  a canonical exact sequence 
\begin{equation}\label{202003082133}
 0 \to \frks_{\leq i} I \to I \to \frks_{> i} I \to 0  
\end{equation} 
in $\sfC(\GrMod A)$. This exact sequence splits if we forget the differentials.

By abuse of notations, 
we denote the functors  
$\frki_{i}: \sfK(\GrInj A) \to \sfK(\Inj \Lambda), \ 
\frks_{i}: \sfK(\GrInj A) \to \sfK(\GrInj A)$. 
The exact sequence \eqref{202003082133}  gives an exact triangle   
\begin{equation}\label{exact triangle}
 \frks_{\leq i} I \to I \to \frks_{> i} I \to 
\end{equation}
in $\sfK(\GrMod A)$ and hence in $\sfD(\GrMod A)$.
 
Let $M$ be an object of $\sfD(\GrMod A)$ and $I \in \sfC(\GrInj A)$ be an injective resolution of $M$, 
that is, $I$ is a DG-injective complex equipped with a quasi-isomorphism  $M \xrightarrow{\sim} I$.  
Then, $\frki_{i} I $ can be computed as 
 \begin{equation}\label{frki description}
 \grRHom_{A}(\Lambda, M)_{i}  \cong \frki_{i} I\textup{ in } \sfD(\Mod \Lambda). 
\end{equation}
Since  $\frki_{i}I$ is a DG-injective complex  of $\Lambda$-modules by injective version of  \cite[Lemma 4.3]{adasore},  
the object $\RHom_{\Lambda}(A, \frki_{i}I) \in \sfD(\GrMod A)$ is represented by the complex 
$\cpxgrHom_{\Lambda}(A, \frki_{i}I) \in \sfC(\GrMod A)$. 
Therefore we have 
\begin{equation}\label{frks description}
\grRHom_{\Lambda}(A, \frki_{i} I)( -i) \cong \frks_{i}I
\end{equation} 
in $ \sfD(\GrMod A)$ (for the convention of gradings see Remark \ref{202003141726}).

\subsubsection{The morphism $\phi_{i}: M \to \grRHom_{\Lambda}(A, M_{i})( -i)$}\label{the morphism phi}

In this Section \ref{the morphism phi} we introduce a morphism $\phi_{i}: M \to \grRHom_{\Lambda}(A, M_{i})( -i)$ 
which is  a key role in the sequel. 

Let $M \in \sfC(\GrMod A)$. 
We define a morphism $\tilde{\phi}_{i}: M \to \cpxgrHom_{\Lambda}(A, M_{i})(-i)$ in $\sfC(\GrMod A)$  in the following way. 
We note that we regard $M_{i}$ as an object in $\sfC(\GrMod \Lambda)$ concentrated in $0$-th degree.  
Let $m \in M_{j}^{n}$ be a homogeneous  element of degree $j$ of $n$-th cohomological degree. 
Then, we define $\tilde{\phi}_{i}(m): A_{i-j} \to M_{i}^{n}$ to be $\tilde{\phi}_{i}(m)(a) := ma$. It is easy to check that 
the morphism $\tilde{\phi}_{i}$ commutes with the differentials of $M$ and $\cpxgrHom_{\Lambda}(A, M_{i})( -i)$. 
We define a morphism $\phi_{i}: M \to \grRHom_{\Lambda}(A, M_{i})( -i)$ in $\sfD(\GrMod A)$ 
to be the composition $\phi_{i}: =\mathsf{can} \circ \tilde{\phi}_{i}$ 
where $\mathsf{can}: \grHom_{\Lambda}(A, M_{i})(-i) \to \grRHom_{\Lambda}(A, M_{i})(-i)$ is a canonical morphism.

\begin{lemma}\label{20170819245}
Let $M$ be an object of $\sfD(\GrMod A)$ and $I \in \sfC(\grInj A)$ an injective resolution of $M$.  
Assume that $M_{>i } = 0$ in $\sfD(\GrMod A)$ for some integer $i \in \ZZ$. 
Then the following assertions hold. 
\begin{enumerate}[(1)]
\item $\frki_{j} I = 0$ in $\sfD(\GrMod A)$ for $j > i$. 

\item $\frks_{\geq  i} I \cong \frks_{i} I$ and $\frks_{j} I = 0$ in $\sfD(\GrMod A)$ for $j > i$. 

\item $\frki_{i} I$ is an injective resolution of $M_{i} \in \sfD(\Mod \Lambda)$. 

\item 
The following diagram is commutative. 
\[
\begin{xymatrix}{
M \ar[rr]^{\phi_{i}\ \ \ \ \ \ } \ar[d]_{\cong}& & \grRHom_{\Lambda}(A, M_{i})(-i) \ar[d]^{\cong}\\
I \ar[rr]_{\mathsf{can}} & & \frks_{ \geq i } I  \cong \frks_{i} I 
}\end{xymatrix}\]
\end{enumerate}

\end{lemma}

\begin{proof}
(1) (2) and  (3)  are proved  as in \cite[Lemma 5.6]{adasore}. 
(4) follows from (3). 
\end{proof}

\begin{corollary}\label{20182242126}
Let $M$ be an object of $\sfD(\GrMod A)$. 
Assume that $M_{>i } = 0$ for some integer $i \in \ZZ$. 
Then for $N\in \sfD(\Mod \Lambda)$, we have 
\[
\RHom_{\GrMod A}(N(-i), M) \cong \RHom_{\Lambda}(N, M_{i})
\]
\end{corollary}

\begin{proof}
Let $I \in \sfC(\GrInj A)$ be an injective resolution of $M$. 
Then we have the following isomorphism in $\sfD(\Mod \kk)$ 
\[
\RHom_{\GrMod A}(N(-i), M) 
\cong \cpxHom_{\GrMod A}(N(-i), I) 
\cong \cpxHom_{\Lambda}(N, \frki_{i} I) 
\cong \RHom_{\Lambda}(N, M_{i})
\]
where the second isomorphism is deduced from Lemma \ref{202003111530} (2) 
and the third isomorphism is deduced from Lemma \ref{20170819245} (3). 
\end{proof}

\section{The Happel functor}\label{the Happel functor}

In this Section \ref{the Happel functor} we recall the Happel functor and related results. 
%
%

\subsection{The Happel functor for a general finitely graded algebra}

Recall that a graded algebra $A = \bigoplus_{i \geq 0} A_{i}$ is called 
 a \emph{graded Noetherian} algebra if  every left or right graded ideal is finitely generated. 
We call a graded algebra $A = \bigoplus_{i=0} A_{i}$  \emph{finitely graded Noetherian} 
if it is graded Noetherian and concentrated in finitely many grading, i.e., $A_{i} = 0$ for $i \gg 0$. 
Recall that we always assume that the maximal degree  
$\ell := \max \{ i \in \NN  \mid A_{i} \neq 0 \}$ of $A$ 
is positive, i.e., $\ell \geq 1$.

In this Section \ref{the Happel functor},   $A = \bigoplus_{i = 0}^{\ell} A_{i}$ denotes a finitely graded Noetherian algebra. 
The subcategory $\grmod A$  of finitely generated graded $A$-module is an abelian subcategory of $\GrMod A$. 
We set $\mod^{[0, \ell -1]} A := \Mod^{[0,\ell -1]} A \cap \grmod A$. 
In other words, $\mod^{[0, \ell -1]} A$ denotes the full subcategory of $\grmod A$ consisting of $M$ 
such that $M_{i} = 0$ for $i \notin [0, \ell -1]$. 

Recall that the singular derived category $\grSing A$ is defined as the Verdier quotient 
$\grSing A:= \sfD^{\mrb}(\grmod A)/ \sfK^{\mrb}(\grproj A)$. 
We denote by $\pi :\sfD^{\mrb}(\grmod A ) \to \grSing A$ the canonical quotient functor. 
Then the Happel functor  is defined in the following way. 

\begin{definition}\label{definition the Happel functor}
We define the Happel functor  $\varpi$ to be the composition of the canonical functors below. 
\[\varpi:  \sfD^{\mrb}(\mod^{[0, \ell-1]} A) \xrightarrow{\ \mathsf{em}_{[0,\ell-1]} \  }\sfD^{\mrb}(\grmod A) \xrightarrow{ \ \pi \ } \grSing A.\]
\end{definition}
We note that the first functor $\mathsf{em}_{[0, \ell -1]}$ is fully faithful by Lemma \ref{embedding lemma}. 

\subsection{The Happel functor for a finitely graded IG-algebra}


We collect definitions  and basic results in the representation theory of  Iwanaga-Gorenstein (IG) algebras. 

Recall that a graded  algebra  $A$ is called Iwanaga-Gorenstein (IG)
 if it is graded  Noetherian  and has finite graded  self-injective dimension on both sides, 
 i.e., $\grinjdim_{A} A < \infty$ and $\grinjdim_{A^{\op}} A < \infty$. 
We remark that a graded algebra $A$ is graded IG if and only if it is IG as an ungraded algebra (see \cite{adasore}).

Let $A$ be a graded IG-algebra. A finitely generated graded $A$-module $M$ is called \textit{graded Cohen-Macaulay} (CM)  
if $\grExt_{A}^{> 0}(M,A) = 0$.  
The graded CM-modules form a full subcategory $\grCM A$ of $\grmod A$ which is a Frobenius category with the induced exact structure. 
The admissible projective-injective objects of $\grCM A$ are finitely generated graded projective $A$-modules. 
Let $\beta'$ be the following composition
\[
\beta': \grCM A \hookrightarrow \grmod A \hookrightarrow \sfD^{\mrb}(\grmod A) \xrightarrow{\pi} \grSing A.
\]
Buchweitz \cite{Buchweitz} and Happel \cite{Happel} proved that 
the functor $\beta'$ descent to  a triangulated equivalence $\beta$ between 
 the stable category $\stabgrCM A = \grCM A/ \grproj A $  to the singular derived category $\grSing A$
 \[
 \beta: \stabgrCM A \xrightarrow{ \cong } \grSing A. 
 \]

As a consequence we obtain a functor mentioned in \eqref{Intoroduction the Happel functor 3}. 
\begin{definition}
We set $\cH : = \beta^{-1} \varpi$ and call it also the Happel functor. 
\[
\cH := \beta^{-1} \varpi: \sfD^{\mrb}(\mod^{[0, \ell -1]} A) \to \stabgrCM A. 
\] 
\end{definition}

Let $\Lambda$ be a finite dimensional algebra. 
Then, the graded  algebra  
$\textup{T}(\Lambda) = \Lambda \oplus \tuD(\Lambda), \ \deg \Lambda = 0, \deg \tuD(\Lambda) =1$  
is graded self-injective and in particular graded IG. 
We have $\grCM \textup{T}(\Lambda)  = \grmod \textup{T}(\Lambda)$, $\mod^{0} A = \mod\Lambda$ and the functor $\cH$ constructed above coincides with  the original Happel functor.

\subsubsection{Iwanaga's Lemma} 

We recall a well-known fact which was first observed by Iwanaga \cite{Iwanaga}. 
Let $A$ be a graded IG-algebra. Then, for a finitely generated graded $A$-module $M$ 
we have $\grpd M < \infty \ \Leftrightarrow \ \grinjdim M < \infty$. 
We give a derived categorical interpretation.

\begin{lemma}\label{derived interpretation of Iwanaga lemma}
Let $A$ be an IG-algebra. 
Then we have 
$\sfK^{\mrb}(\grproj A) = \sfD^{\mrb}(\grmod A) \cap \sfK^{\mrb}(\GrInj A)$.
\end{lemma}

\begin{proof}
For simplicity, we set $ \sfI := \sfD^{\mrb}(\grmod A) \cap \sfK^{\mrb}(\GrInj A)$. 
It follows from $\grinjdim A < \infty$  that 
$\sfK^{\mrb}(\grproj A) \subset \sfI$. 

Let $M$ be an object of $\sfI$ 
and $ n \in \ZZ$ be an integer such that $\tuH^{< n}(M)= 0$. 
We take a projective resolution $P\in \sfC^{-,\mrb}(\grproj A)$ of $M$ 
and its brutal truncations $\sigma^{> n} P, \ \sigma^{\leq n} P$. 
\[
\begin{split}
\sigma^{ > n}P &:   \cdots \to  0 \xrightarrow{\ \ \ \ \ \ \ \ \ \ \  } 0 \xrightarrow{ \ \ \ \ } 
 P^{n+1} \xrightarrow{\partial_{P}^{n +1}}  P^{n +2} \to \\
\sigma^{\leq n}P& :  \cdots \to P^{n-1} \xrightarrow{\partial_{P}^{ n-1}} P^{n} \xrightarrow{ \ \ \ \ } 
 0 \xrightarrow{ \ \ \ \ \ \ \ \ \ }  0 \to \cdots 
\end{split}
\]
If we set $N := \Coker \partial_{P}^{n-1}$, then $\sigma^{\leq n}P \cong N[-n]$ in $\sfD(\grmod A)$ 
and we obtain an exact triangle $ \sigma^{> n} P \to M \to N[-n] \to $ in $\sfD^{\mrb}(\grmod A)$. 
Observe that  $\sigma^{> n} P$ belongs to $\sfK^{\mrb}(\grproj A)$. 
Hence $M$ and $\sigma^{ > n} P$ belong to $\sfI$. 
Therefore $N$ belongs to $\sfI$ and hence $\grinjdim N < \infty$.   Consequently, we have  $\grpd N < \infty$.  
 It follows from $\sigma^{ > n} P, N \in \sfK^{\mrb}(\grproj A)$ 
 that $M \in \sfK^{\mrb}(\grproj A)$. 
\end{proof}

\subsection{Quasi-Veronese algebra construction}\label{Quasi-Veronese algebra construction}

The reader can postpone Section \ref{Quasi-Veronese algebra construction} until the proof of Theorem \ref{main theorem 2}.

Let $A= \bigoplus_{i= 0}^{\ell} A_{i}$ be a finitely graded Noetherian algebra. 
We recall the quasi-Veronese algebra construction and the relationship with the Beilinson algebra $\nabla A$ defined in \eqref{Beilinson algebra} of a graded algebra from \cite{Mori B-construction} 
(see also \cite{adasore}). 

%
%

We may regard the Happel functor $\varpi$ to be a functor $\sfD^{\mrb}(\mod \nabla A) \to \grSing A$ 
via the equivalence $\sfq : \sfD^{\mrb}(\mod^{[0, \ell -1]} A) \cong \sfD^{\mrb}( \mod  \nabla A)$ of \eqref{Introduction nabla}. 
\[
\varpi : \sfD^{\mrb}(\mod \nabla A) \xrightarrow{ \sfq^{-1} \ \cong \ } \sfD^{\mrb}(\mod^{[0, \ell-1]}A) \to \grSing A.  
\]

We define 
a bimodule $\Delta A$ over $\nabla A$ to be 
\[
\Delta A: = 
\begin{pmatrix} 
A_{\ell} & 0 & \cdots & 0 \\
A_{\ell- 1} & A_{\ell} & \cdots  &0 \\
\vdots & \vdots     &        & \vdots \\
A_{1} & A_{2}   & \cdots  & A_{\ell}
\end{pmatrix} 
\]
where the bimodule structure are 
given by matrix  multiplications. 
Then, 
the trivial extension algebra 
 $\nabla A \oplus \Delta A$ 
with the grading $\deg \nabla A = 0, \deg  \Delta A = 1$ 
is  nothing but the $\ell$-th quasi-Veronese algebra $A^{[\ell]}$ of $A$ introduced in \cite{Mori B-construction}.  
An important fact shown in \cite{Mori B-construction} is that there exists a $\kk$-linear equivalence 
$\mathsf{qv} :\GrMod A \xrightarrow{\cong} \GrMod A^{[\ell]}$ 
such that $(1) \mathsf{qv} = \mathsf{qv} (\ell)$. 
It follows that  $A$ is graded Noetherian (resp. IG)  if and only if so is $A^{[\ell]}$. 
Moreover the equivalence $\mathsf{qv}$ induces  
 equivalences of $\kk$-linear categories 
$\mathsf{qv} : \grmod A \xrightarrow{\cong} \grmod A^{[\ell]}$   
and 
$\mod^{[0, \ell -1]} A \cong \mod^{0} A^{[\ell]} \cong \mod \nabla A$. 
The equivalence $\mathsf{qv}$ induces an equivalence between the singular derived categories as well as 
 the following commutative diagram.  
\begin{equation}\label{qv diagram}
\begin{xymatrix}@R = 10pt{ 
&&\sfD^{\mrb}(\mod^{[0, \ell -1]} A) \ar[rr]^{\varpi_{A}} \ar@{-}[dll]_{\cong}^{\sfq_{A}} && \grSing A \ar[dd]_{\cong}^{\mathsf{qv}} \\
\sfD^{\mrb}(\mod\nabla A)  && && \\
&& \sfD^{\mrb}(\mod^{0} A^{[\ell]}) \ar[rr]_{\varpi_{A^{[\ell]}} }  \ar@{-}[ull]^{\cong}_{\sfq_{A^{[\ell]}}}&& \grSing A^{[\ell]} 
}\end{xymatrix}
\end{equation}

Thanks to results above, we may reduce representation theoretic problems of a finitely graded algebras $A = \bigoplus_{i = 0}^{\ell}A_{i}$ to the case where maximal degree $\ell =1$. 
A finitely graded algebra $A = A_{0} \oplus A_{1}$ of $\ell =1$ 
can be regarded as the trivial extension algebra $A = \Lambda \oplus C$ of $\Lambda := A_{0}$ by $C := A_{1}$ with the grading $\deg \Lambda = 0, \deg C = 1$. 
We point out that in this case, we have $T$ of \eqref{canonical construction} is $\Lambda$. 

\section{Homologically well-graded algebras}\label{Homologically well-graded algebras} 

In Section \ref{Homologically well-graded algebras}, we introduce homologically well-graded algebras. 
We give their characterization and show that the Happel functors of them are fully faithful.

Let $A= \bigoplus_{i= 0}^{\ell} A_{i}$ be a finitely graded algebra with the maximal degree $\ell = \max\{ i \mid A_{i} \neq 0\}$. 
For simplicity we set $\Lambda:= A_{0}$.

\subsection{Homologically well-graded complexes}

Now we  introduce a notion which plays a central role in this paper. 

\begin{definition}
Let $i\in \ZZ$ be  an integer.  
An object $M \in \sfD(\GrMod A)$ is called $i$\textit{-homologically well-graded} ($i$-hwg)
if $\grRHom_{A}(\Lambda, M)_{j} = 0$ in $\sfD(\Mod \Lambda)$  for $j \neq i$. 

An object $M \in \sfD(\GrMod A)$ is called \textit{homologically well-graded}(hwg)
if it is $i$-homologically  well-graded for some $i\in \ZZ$. 
\end{definition}

We collect equivalent conditions for homologically well-gradedness.

\begin{lemma}\label{201708191905}
Let   $M$ be an object of $\sfD(\GrMod A)$ and $I \in \sfC(\GrInj A)$ an injective resolution of $M$. 
Then, 
for an integer $i \in \ZZ$, the following conditions are equivalent. 

\begin{enumerate}[(1)]

\item $M$ is $i$-homologically well-graded. 
 
 \item  $\frki_{j}I = 0$ in $\sfD(\Mod \Lambda)$ for $j \neq i$. 
 
 \item $\frks_{i} I  \cong I$ in $\sfD(\GrMod A)$. 
 
 \item The morphism $\phi_{i}: M \to \grRHom_{\Lambda}(A, M_{i})(-i)$ defined in Section \ref{the morphism phi} is an isomorphism in $\sfD(\GrMod A)$.
 
 \item 
 $\sfD( \Mod^{\ZZ \setminus {\{ i\} }} A) \perp M $.


 \end{enumerate}
\end{lemma}

\begin{proof}
(1) $\Leftrightarrow$ (2) follows from the isomorphism \eqref{frki description}. 

(2) $\Rightarrow$ (3) follows from \eqref{exact triangle} and  \eqref{frks description}. 

(3) $\Rightarrow$ (4). 
Since $M \cong I \cong \frks_{i} I$ in $\sfD(\GrMod A)$, we have $M_{> i} =0$ in $\sfD(\GrMod A)$.  
It  follows from Lemma \ref{20170819245}  
that the morphism $\phi_{i}$ is an isomorphism. 

(4) $\Rightarrow$ (5) follows from the isomorphism below for $N \in \sfD(\GrMod A)$. 
\begin{equation*}\label{20182242116}
\RHom_{\GrMod A}(N, \grRHom_{\Lambda}(A, M_{i})(-i)) \cong \RHom_{\GrMod \Lambda}(N(i), M_{i})
\end{equation*}


(5) $\Rightarrow$ (2) follows from the isomorphism 
$\RHom_{\GrMod A} (\Lambda(-j), M) 
\cong \grRHom_{A}(\Lambda, M)_{j}$. 
\end{proof}

Combining Lemma \ref{20170819245}.(3) and Lemma \ref{201708191905}, 
we see that the injective dimension of an $i$-hwg object $M \in \sfD(\GrMod A)$ 
coincides with that of $M_{i} \in \sfD(\Mod \Lambda)$. 
For the injective dimension of an object of derived category, 
we refer \cite{Avramov-Foxby}. 
We note that in the case where  $M \in \GrMod A$, 
the injective dimension of $M$ as an object of $\sfD(\GrMod A)$ 
coincides  with the usual injective  dimension as a graded $A$-module.

\begin{corollary}\label{201810022155} 
Let $M \in \sfD(\GrMod A)$ be an $i$-hwg object for some $i \in  \ZZ$. 
Then, we have 
\[
\grinjdim_{A}M = \injdim_{\Lambda} M_{i}. 
\]
\end{corollary}

\subsection{Homologically well-graded algebras}\label{subsubsection: hwg algebra}

We call a finitely graded algebra $A = \bigoplus_{i = 0}^{\ell} A_{i}$
\textit{right  strictly well-graded (right swg)} if  the following condition is satisfied: 
\[
\grHom_{A}(\Lambda, A)_{j} = 0  \textup{  unless   } j = \ell.
\]
In other words, $A$ is right swg if and only if  the degree $\ell$-part $A_{\ell} $ is an essential $A$-submodule of $A$. 
Observe that in the case where $A$ is a finite dimensional algebra over some field, 
$A$ is right swg if and only if $\soc A \subset A_{\ell}$. 
We note that in several papers, such  a finitely graded algebra  is said to have Gorenstein parameter. 
  
We call a finitely graded algebra $A = \bigoplus_{i = 0}^{\ell} A_{i}$
\textit{left  strictly well-graded (left swg)} if the opposite algebra $A^{\op}$ is right hwg.

We call  a finitely graded algebra $A$ \emph{strictly well-graded} if it is both right and left swg. 

In  \cite{Chen trivial}, Chen called a finite dimensional graded algebra $A= \bigoplus_{i = 0}^{\ell}A_{i}$ \emph{right well-graded} 
if $eA_{\ell} \neq 0$ for any primitive idempotent element $ e \in A_{0}$. 
It is easy to see that a right swg algebra is right well-graded but the converse does not hold in general.

\begin{definition}\label{definition hwg}
A finitely graded algebra $A$ is called \textit{right homologically well-graded} (right hwg)  
if $A_{A} \in \sfD(\GrMod A)$ is homologically well-graded.

$A$ is called \textit{left homologically  well-graded}(left hwg) if $A^{\op}$ 
is right  homologically  well-graded.  

$A$ is called \textit{homologically  well-graded}(hwg) if it
is left and right hwg.  
\end{definition}

We remark that a finitely graded algebra $A = \bigoplus_{i = 0}^{\ell} A_{i}$ is right hwg if and only if 
the object $A \in \sfD(\GrMod A)$ is an $\ell$-hwg object. 
We also remark that a right hwg algebra is right swg.

We give characterizations of a right hwg algebra. 

For an integer $i \in \ZZ$, we set $\proj^{< 0 \textup{-gen}} A := \add\{ A( -i) \mid i < 0\}$ and 
$\proj^{\geq 0 \textup{-gen}} A := \add\{ A( -i) \mid i \geq  0\}$.

\begin{proposition}\label{proposition hwg algebra}
Let $A = \bigoplus_{i= 0}^{\ell} A_{i} $ be a finitely graded algebra 
and $\ell := \max\{ i \mid A_{i } \neq 0\}$.  
Then the following conditions are equivalent. 
\begin{enumerate}[(1)] 
\item $A$ is right hwg. 




\item The canonical morphism $A \to \grRHom_{\Lambda}(A, A_{\ell})(-\ell)$ is an isomorphism in $\sfD(\GrMod A)$. 

\item 
$\sfD(\Mod^{\leq \ell-1} A) \perp \sfK^{\mrb}(\proj^{\geq 0 \textup{-gen}}A)$.

\item $A$ is right swg  and $\grExt_{A}^{>0}(A_{\leq k}, A) = 0$ for $k= 0, \cdots, \ell -1$. 

\item $A$ is right swg and $\grExt_{A}^{>0}(\Lambda, A) = 0$. 
\end{enumerate}  
\end{proposition}

\begin{proof}
Throughout the proof $I$ denotes the graded injective resolution of $A$. 
It follows from Lemma \ref{201708191905} that the conditions (1), (2) are equivalent. 

(1) $\Rightarrow$ (3). 
$\sfD(\Mod^{\ZZ \setminus{ \{\ell +i \} }} A) \perp A(-i)$ for $i \in \ZZ$ by Lemma \ref{201708191905}.  
 Since $\sfK^{\mrb}(\proj^{\geq 0 \textup{-gen}} A )= \thick \{ A(-i) \mid i \geq 0\}$, we conclude that the condition (3) holds.

(3) $\Rightarrow $ (1). 
 We prove $\frks_{i} I = 0$ unless $i = \ell$. 
By Lemma \ref{20170819245}, we have $\frks_{> \ell} I = 0$.  
The condition (3) implies that  $\frki_{i} I = \grRHom_{A}(\Lambda, A)_{i} = 0$ for $ i \leq \ell -1$.  
Hence $\frks_{i}I = 0$ for  $  i \leq \ell -1$.

(1) $\Rightarrow$ (4). 
As is mentioned before right homologically well-gradedness implies right strictly well-gradedness.   
Since $\frks_{\ell} I \cong I$ by Lemma \ref{201708191905} (3),  we have the following isomorphisms  for $k= 0, \cdots, \ell -1$ in $\sfD(\GrMod A)$ by Lemma \ref{202003082113} (1)
\[
\grRHom_{A}(A_{\leq k}, A) \cong \grHom^{\bullet}_{A}(A_{\leq k}, I) \cong I_{\geq \ell -k} \cong A_{\geq \ell -k}. 
\]
In particular, we have $\grExt_{A}^{> 0}(A_{\leq k}, A) = \tuH^{>0}(A_{\geq \ell -k}) = 0$. 

The implications (4) $\Rightarrow$  (5) is clear. 
The implication (5) $\Rightarrow$ (1) follows from an isomorphism $\tuH^{n}(\grRHom_{A}(\Lambda, A)) \cong  \grExt_{A}^{n}(\Lambda, A)$ for $n \geq 0$. 
\end{proof}
%
%

\begin{corollary}\label{201708260214}
Assume that $\kk$ is a field and  a graded algebra  $A= \bigoplus_{i = 0}^{\ell}A_{i}$ is  finite dimensional and self-injective. 
Then $A$  is right well-graded if and only if it is right swg if and only if it is right  hwg.
\end{corollary}

\begin{remark}
Assume that $\kk$ is a field and  a graded algebra  $A= \bigoplus_{i = 0}^{\ell}A_{i}$ is  finite dimensional and self-injective. 
By \cite[Lemma 2.2]{Chen trivial}, $A$ is right well-graded if and only if it is left well-graded. 
Therefore, by Corollary \ref{201708260214}, $A$ is right hwg if and only if it is left hwg, if and only if it is hwg. 
\end{remark}

We leave the proof of the following lemma to the readers. 

\begin{lemma}\label{201708232231}
A finitely graded algebra $A = \bigoplus_{i = 0}^{\ell} A_{i}$ is 
right hwg (resp. left hwg) 
if and only if 
so is $A^{[\ell]}$.
\end{lemma}

\subsection{Homologically well-gradedness and the Happel functor}

Let $A$ be a finitely graded Noetherian algebra,  so that we have the Happel functor $\varpi: \sfD^{\mrb}(\mod^{[0,\ell -1]} A) \to \grSing A$. 
Homologically well-gradedness guarantees that the Happel functor $\varpi$ is fully faithful.

\begin{proposition}\label{201708231410}
If a finitely graded Noetherian algebra $A$ is right hwg, 
then  the Happel functor $\varpi$ is fully faithful. 
\end{proposition}

This proposition is a partial generalization of Orlov's result given in \cite{Orlov}.
In view of Proposition \ref{proposition hwg algebra}, the proof is essentially the same with  Orlov's original proof. 
For the convenience of  the readers, we provide the whole proof.

\begin{proof}
Since $\Hom_{\GrMod A}(P, M) = 0$ for $P \in \proj^{< 0 \textup{-gen}} A, M \in \mod^{\geq 0} A$, 
we have $\sfK^{\mrb}(\proj^{<0\textup{-gen}}A) \perp \sfD^{\mrb}(\mod^{\geq 0} A)$. 
Thus in particular
we have  
$\sfK^{\mrb}(\proj^{<0\textup{-gen}}A) \perp  \left( \sfD^{\mrb}(\mod^{[0,\ell -1]} A)*\sfK^{\mrb}(\proj^{\geq 0\textup{-gen}}A)\right)$. 
Therefore, the composition $\widetilde{\varpi}$ of canonical functors below  is fully faithful by \cite[Proposition II.2.3.5]{Verdier}. 
\[
\begin{split}
\widetilde{\varpi}:  
\sfD^{\mrb}(\mod^{[0,\ell -1]} A) * \sfK^{\mrb}(\proj^{\geq 0\textup{-gen}} A)  
&  \hookrightarrow \sfD^{\mrb}(\grmod A) \\
& \xrightarrow{ \ \mathsf{qt}_{1} \ }  \sfD^{\mrb}(\grmod A)/\sfK^{\mrb}(\proj^{<0 \textup{-gen}}A).  
\end{split}
\]

Since $A$ is hwg, we have $\sfD^{\mrb}(\mod^{[0,\ell-1]} A) \perp \sfK^{\mrb}(\proj^{\geq 0 \textup{-gen}} A)$ by Proposition \ref{proposition hwg algebra}. 
Hence by \cite{Verdier} again, 
if we denote the following composition by $F$, then it  is fully faithful 
\[
\begin{split}
F: \sfD^{\mrb}(\mod^{[0,\ell -1]} A) 
 &\to \sfD^{\mrb}(\grmod A)/\sfK^{\mrb}(\proj^{<0 \textup{-gen} }A) \\
 & \xrightarrow{ \ \mathsf{qt}_{2} \ } 
 \Bigl( \sfD^{\mrb}(\grmod A)/\sfK^{\mrb}(\proj^{<0 \textup{-gen}}A) \Bigr)/\widetilde{\varpi}( \sfK^{\mrb}(\proj^{\geq 0\textup{-gen}} A))
\end{split} 
\]
where the first arrow is the restriction 
$\widetilde{\varpi}|_{ \sfD^{\mrb}(\mod^{[0,\ell -1]} A)}$ of $\widetilde{\varpi}$ 
and the second arrow is the quotient functor.

Observe that the kernel $\Ker( \mathsf{qt}_{2} \mathsf{qt}_{1})$ of the composition $\mathsf{qt}_{1}\mathsf{qt_{2}}$ of the quotient functors $\mathsf{qt}_{1}, \mathsf{qt_{2}}$ 
is $\sfK^{\mrb}(\proj^{<0\textup{-gen}}A)* \sfK^{\mrb}(\proj^{\geq 0\textup{-gen}}A) = \sfK^{\mrb}(\grproj A)$. 
 It follows from the universal properties of the quotient functors that 
there exists an equivalence $G$ which completes the following commutative diagram.  
\[
\begin{xymatrix}@C=15pt{
\sfD^{\mrb}(\grmod A) \ar[rr]^-{\mathsf{qt}_{1}} \ar[d]_{\pi} &&
\sfD^{\mrb}(\grmod A)/\sfK^{\mrb}(\proj^{<0 \textup{-gen}}A) \ar[d]^{\mathsf{qt}_{2}} \\
\grSing A = \sfD^{\mrb}(\grmod A) /\sfK^{\mrb}(\grproj A)  && 
\Bigl( \sfD^{\mrb}(\grmod A)/\sfK^{\mrb}(\proj^{<0 \textup{-gen}}A) \Bigr)/\widetilde{\varpi}( \sfK^{\mrb}(\proj^{\geq 0\textup{-gen}} A))
\ar@{-->}[ll]_-{ \cong }^-{G} 
}\end{xymatrix}
\]
Therefore the composition $GF :\sfD^{\mrb}(\mod^{[0,\ell -1]} A) \to \grSing A$ is naturally isomorphic to the Happel functor $\varpi$. 
Since $F$ is fully faithful, we conclude that $\varpi$ is fully faithful as desired. 
\end{proof}

\section{A characterization of homologically well-graded IG-algebras}\label{characterization}

In this Section \ref{characterization}, we give a characterization of hwg IG-algebras from  a view point of self-duality. 
We start by recalling the definition of  a cotilting bimodule and its important property from  \cite{Miyachi}.  

\begin{definition}\label{definition of cotilting modules} 
Let $\Lambda$ be a Noetherian algebra. 
A $\Lambda$-$\Lambda$-bimodule $C$ is called \textit{cotilting} 
if the following conditions are satisfied. 
\begin{enumerate}[(1)] 
\item $C$ is finitely generated as both a right $\Lambda$-module 
and a left $\Lambda$-module.  

\item 
$\injdim\limits_{\Lambda} C< \infty, \injdim\limits_{\Lambda^{\op}} C < \infty$. 

\item 
$\Ext_{\Lambda}^{>0}(C,C) = 0, \ \Ext_{\Lambda^{\op}}^{>0}(C,C) = 0$. 

\item 
The natural algebra morphism $\Lambda  \to \Hom_{\Lambda}(C, C)$ 
is an isomorphism.

The natural algebra morphism $\Lambda^{\op} \to \Hom_{\Lambda^{\op}}(C, C)$ 
is an isomorphism.
\end{enumerate}
\end{definition}

We point out several well-known facts.
A Noetherian algebra $\Lambda$ is  IG if and only if $\Lambda$ is a cotilting bimodule over $\Lambda$. 
Moreover, over an IG-algebra $\Lambda$, cotilting bimodules are the same notion with tilting bimodules.

\begin{theorem}[{Miyachi \cite[Corollary 2.11]{Miyachi}}]\label{cotilting theorem}
Let $\Lambda$ be a Noetherian algebra and $C$ be a cotilting bimodule over $\Lambda$. 
Then $\RHom(-, C)$ induces an equivalence of triangulated categories. 
\[
\RHom_{\Lambda}(-, C) : 
\sfD^{\mrb}(\mod \Lambda) \cong \sfD^{\mrb}(\mod \Lambda^{\op})^{\op}: \RHom_{\Lambda^{\op}}(-,C). 
\]
\end{theorem}

The following theorem gives a characterization of hwg IG-algebra. 


\begin{theorem}\label{main theorem}
Let $A= \bigoplus_{i = 0}^{\ell}A_{i}$ be a  finitely graded Noetherian  algebra with $\Lambda := A_{0}$. 
Then the following conditions are  equivalent. 

\begin{enumerate}[(1)]
\item $A$ is a hwg IG-algebra.

\item $A$ is a right hwg IG-algebra. 

\item $A$ is a right swg IG-algebra and the module $T$ is CM. 

\item The following conditions are satisfied. 

\begin{enumerate}[(4-i)]
\item  $A_{\ell}$ is a cotilting bimodule over $\Lambda$ 

\item 
There exists a $\Lambda$-$A$-bimodule isomorphism 
\[
\alpha: A \cong \grHom_{\Lambda}(A, A_{\ell})(-\ell)
\]

\item $\grExt_{\Lambda}^{> 0}(A, A_{\ell}) = 0$. 
\end{enumerate}
\end{enumerate}
If these conditions are satisfied,  then $\injdim_{A} A = \injdim_{\Lambda} A_{\ell}$.  
\end{theorem}

\begin{remark}\label{remark to main theorem}
The condition (4-ii), (4-iii) are summarized to the condition: 

(4-ii+iii) There exists an isomorphism 
\[
\hat{\alpha}: A \cong \grRHom_{\Lambda}(A, A_{\ell})(-\ell)
\]
in the derived category of $\Lambda$-$A$-bimodules. 
\end{remark}

\begin{proof}
The implication (1) $\Rightarrow$ (2) is clear. 
The equivalence (2)  $\Leftrightarrow$ (3) follows from Proposition \ref{proposition hwg algebra}.

We prove (2) $\Rightarrow$ (4). 
The conditions (4-ii) and (4-iii) follows from (4) of Lemma \ref{201708191905}. 
It remains to show that $A_{\ell}$ is a cotilting bimodule over $\Lambda$.  

As $A$ is Noetherian, $A_{\ell}$ is finitely generated on both sides over $\Lambda$.  
Let $I\in \sfC^{\mrb}(\GrInj A)$ be a graded injective resolution of $A$. 
By Lemma \ref{20170819245}.(3), $\frki_{\ell} I$ is a $\Lambda$-injective resolution of $A_{\ell}$. 
Moreover, it is bounded since $I$ is bounded.     
Therefore, $\injdim_{\Lambda} A_{\ell}  < \infty$. 
Similarly, we have $\injdim_{\Lambda^{\op}} A_{\ell} < \infty$. 

Looking at the degree $0$-part of the isomorphism $\phi_{\ell}: A \cong \grRHom_{\Lambda}(A, A_{\ell})( -\ell)$,  
we obtain an isomorphism $\Lambda \cong \RHom_{\Lambda}(A_{\ell}, A_{\ell})$.

We set $(-)^{*} := \grRHom_{A}(-,A):  \sfD^{\mrb}(\grmod A) \to \sfD^{\mrb}(\grmod A^{\op})$. 
We claim that  $ (\Lambda(-\ell))^{*} \cong A_{\ell}$.  
Indeed, since $A$ is right hwg, we have $(\Lambda(-\ell))^{*}_{i} = \grRHom_{A}(\Lambda, A)_{i+ \ell} = 0$ for $i \neq 0$. 
By Corollary \ref{20182242126}, we have \[
(\Lambda(-\ell))^{*}_{0} = \RHom_{\GrMod A}(\Lambda( -\ell), A) \cong \RHom_{\Lambda}(\Lambda, A_{\ell}) = A_{\ell}.
\]
This finishes the proof of claim.

As we remarked above, $A$ is a cotilting module over $A$. 
Therefore, by a graded version of Theorem \ref{cotilting theorem}, 
the functor $(-)^{*} := \grRHom_{A}(-,A)$ give a contravariant equivalence 
from $\sfD^{\mrb}(\grmod A)$ to $\sfD^{\mrb}(\grmod A^{\op})$. 
Therefore we obtain an isomorphism $\Lambda \cong \RHom_{\Lambda^{\op}} (A_{\ell}, A_{\ell})$ as follows 
\[
\begin{split}
\Lambda 
 \cong \RHom_{\GrMod A }(\Lambda(-\ell), \Lambda(-\ell)) 
  \xrightarrow{ \ \ \cong \ (-)^{*} } \RHom_{\GrMod A^{\op}}(A_{\ell}, A_{\ell})  \cong  \RHom_{\Lambda^{\op}}(A_{\ell}, A_{\ell} ). 
\end{split} 
\]
This shows that $A_{\ell}$ is a cotilting bimodule over $\Lambda$ as desired.

(4) $\Rightarrow$ (1). 
First we prove $\injdim\nolimits_{A} A < \infty$. 
By (4-ii) 
it is enough to show that  the graded $A$-module $\grHom_{\Lambda}(A, A_{\ell})$ has finite injective dimension. 
It is easy to see that 
if  $J$ is an injective $\Lambda$-module, 
then  the graded $A$-module $\grHom_{\Lambda}(A,J)$ is a graded injective $A$-module. 
We take an injective resolution $J^{\bullet} $ \[
 0 \to A_{\ell} \to J^{0} \to J^{1} \to \cdots \to J^{d} \to 0 
\]
 of $A_{\ell}$ as a $\Lambda$-module 
where $d := \injdim_{\Lambda}A_{\ell}$ is finite by (4-i). 
By the assumption (4-iii), we obtain an injective resolution 
\[
 0 \to\grHom_{\Lambda}(A,A_{\ell})  \to \grHom_{\Lambda}(A,J^{0})  
\to \grHom_{\Lambda}(A, J^{1} )  \to \cdots \to \grHom_{\Lambda}(A,J^{d})  \to 0.  
\]
of $\grHom_{\Lambda}(A, A_{\ell})$.

Next  we prove that $A$ is right hwg. 
By Lemma \ref{201708191905},  it suffices to show that  the canonical morphism $\phi_{\ell}$ is an isomorphism. 
By (4-iii),  it is enough to show that the $0$-th cohomology morphism 
$\tuH^{0}(\phi_{\ell}): A \to \grHom_{\Lambda}(A, A_{\ell})( -\ell)$ is an isomorphism.  
We set $\beta := \tuH^{0}(\phi_{\ell})$. 
For notations simplicity,  we set  $\alpha_{a} := \alpha(a), \beta_{a}:= \beta(a)$ for a homogeneous  element $a \in A_{i}$ of degree $i$. 
We note that $(\grHom_{\Lambda}(A, A_{\ell})( -\ell ))_{i} = \Hom_{\Lambda}(A_{\ell - i}, A_{\ell})$. 
Hence $\alpha_{a}$ and $\beta_{a}$ can be regarded as   elements of $\Hom_{\Lambda}(A_{\ell -i}, A_{\ell})$. 

We claim that 
for  $a \in A_{0}, b \in A_{i}$  
we have an  equality 
$\alpha_{a} \circ \beta_{b} = \alpha_{ab}$ in   $\Hom_{\Lambda}(A_{\ell -i}, A_{\ell})$ 
where $\alpha_{a} \circ \beta_{b}$ denotes the composition 
$A_{\ell -i} \xrightarrow{ \beta_{b}} A_{\ell} \xrightarrow{ \alpha_{a}} A_{\ell}$.  
Indeed, by the definition of the morphism $\phi_{\ell}$, we have
$\beta_{b}(c) = bc$ for  $c \in A_{\ell -i}$. 
It follows that  $(\alpha_{a}\circ \beta_{b})(c) = \alpha_{a}(bc) =(\alpha_{a}\cdot b)(c)$  
where $ - \cdot b$ denotes the right action of $b \in A$ on $\grHom_{\Lambda}(A, A_{\ell})( -\ell)$. 
Therefore we have $\alpha_{a} \circ \beta_{b} = \alpha_{a}\cdot b$ in  $\grHom_{\Lambda}(A, A_{\ell})( -\ell)$. 
On the other hand, since $\alpha$ is a right $A$-module homomorphism, we have $\alpha_{a}\cdot b= \alpha_{ab}$.
Thus we conclude $\alpha_{a} \circ \beta_{b} = \alpha_{ab}$ as desired.

In a similar way, we can prove 
an equality $ \beta_{a}\circ \alpha_{b} = \alpha_{ab}$ for a
for  $a \in A_{0}, b \in A_{i}$. 

Since $\alpha$ gives an isomorphism $A_{0} \cong \Hom_{\Lambda}(A_{\ell}, A_{\ell})$, 
there exists an element $a \in A_{0}$ such that $\alpha_{a} = \id_{A_{\ell}}$. 
It follows that $\id_{A_{\ell}} = \alpha_{a}  = \beta_{a} \circ \alpha_{1_{A}} = \alpha_{1_{A}} \circ  \beta_{a}$. 
This shows that $\alpha_{1_{A}}$ is an isomorphism. 

It follows from the claim that we have $\alpha_{1_{A}} \circ \beta_{b} = \alpha_{b}$ for $ b\in A$. 
Hence, we have $\beta_{b} = \alpha_{1_{A}}^{-1}\circ\alpha_{b}$. 
In other words,  $\beta$ is obtained as the following composition 
\[
\beta: 
A \xrightarrow{ \ \alpha \ } 
\grHom_{\Lambda}(A, A_{\ell})(-\ell)  \xrightarrow{ \ \alpha_{1_{A}}^{-1} \circ - \ } \grHom_{\Lambda}(A, A_{\ell})(- \ell).
\]
This shows that $\beta$ is an isomorphism as desired. 
We finished the proof that $A$ is right hwg. 

The condition (4-i) is right-left symmetric by the definition of cotilting bimodules. 
The condition (4-ii) and (4-iii) are  right-left symmetric by Theorem \ref{cotilting theorem}
 and Remark \ref{remark to main theorem}.     
Hence, we deduce that $\injdim_{A^{\op}}A < \infty $ and $A$ is left hwg. 
\end{proof}

As an example, we show that a hwg self-injective algebra is nothing but a graded Frobenius algebra. 
\begin{example}\label{Example graded Frobenius}
In this Example \ref{Example graded Frobenius}, 
we assume that $\kk$ is a field and $A$ is finite dimensional over $\kk$.
Recall that 
a finite dimensional graded algebra $A = \bigoplus_{i = 0}^{\ell} A_{i}$ 
is called \textit{graded Frobenius} 
if there exists an isomorphism of graded (right) $A$-modules 
\begin{equation}\label{Introduction graded Frobenius}
\tuD(A)(-\ell) \cong A. 
\end{equation}
It is clear that a graded Frobenius algebra is self-injective. 

Observe that 
since we have a canonical isomorphism $\tuD(A) \cong  \grHom_{A_{0}}(A, \tuD(A_{0}))$, 
the defining isomorphism \eqref{Introduction graded Frobenius} is written as 
\begin{equation}\label{Introduction graded Frobenius 2}
\grHom_{A_{0}}(A, \tuD(A_{0}))(-\ell) \cong A. 
\end{equation}
Since $\tuD(A_{0})$ is a cotilting bimodule over $A_{0}$, 
we see that a graded Frobenius algebra is a hwg self-injective algebra. 

On the other hand, 
if $A$ is a hwg self-injective algebra, then  $A_{\ell}$ is a cotilting bimodule of  injective dimension $0$ on both sides by Corollary \ref{201810022155}. 
It follows that $(A_{\ell})_{A_{0}} \cong \tuD(A_{0})_{A_{0}}$. 
Thus we obtain the isomorphism \eqref{Introduction graded Frobenius 2} and hence  the isomorphism \eqref{Introduction graded Frobenius}.
\end{example}

In the case $\ell =1$, 
a graded algebra  $A = A_{0} \oplus A_{1}$ is regarded as the trivial extension algebra $A = \Lambda \oplus C$ of an algebra  $\Lambda = A_{0}$ by a bimodule $C = A_{1}$ over $\Lambda$. 

\begin{corollary}\label{ell = 1 corollary}
Let $\Lambda$ be a Noetherian algebra and  $C$ be a $\Lambda$-$\Lambda$-bimodule 
which is finitely generated on both side. 
Then, the trivial extension algebra $A = \Lambda \oplus C$ with the grading $\deg \Lambda = 0, \deg C = 1$ 
is a hwg IG-algebra 
if and only if $C$ is a cotilting bimodule. 
\end{corollary} 

\begin{proof}
The ``only if" part is a direct consequence of Theorem \ref{main theorem}. 
Conversely, if $C$ is a cotilting bimodule, then it is immediately check the condition (4) of Theorem \ref{main theorem}. 
\end{proof}

If $A = \Lambda \oplus C$ is hwg IG, 
then the dualities $\RHom_{\Lambda}(- ,C)$ and $\grRHom_{A}(-, A)$ are compatible under Happel functor. 

\begin{proposition}\label{compatibility of dualities}
Let $A = \Lambda \oplus C$ be a hwg IG-algebra. 
Then the following diagram is commutative. 
\[
\begin{xymatrix}{ 
\sfD^{\mrb}(\mod \Lambda ) \ar[rr]^{\varpi} \ar[d]_{\RHom_{\Lambda}(-,C)} && \grSing A \ar[d]^{\grRHom_{A}(-, A(1))} \\ 
\sfD^{\mrb}(\mod \Lambda^{\op} ) \ar[rr]_{\varpi}  && \grSing A^{\op} 
}\end{xymatrix}\]
\end{proposition}

We note that both vertical arrows are equivalence functors. 

\begin{proof}
Let $M \in \sfD^{\mrb}(\mod \Lambda)$. 
First note that $A(1)$ is $0$-hwg object of $\sfD(\GrMod A)$. 
By Lemma \ref{201708191905}(5), we have $\grRHom_{A}(M,A(1))_{ i} = 0$ for $ i \neq 0$. 
Therefore,  we have an isomorphism 
$\grRHom_{A}(M,A(1)) \cong  \RHom_{\GrMod A}(M, A(1))$ of objects of $\sfD(\GrMod A^{\op})$. 
By Corollary \ref{20182242126}, we have 
$\RHom_{\GrMod A}(M, A(1)) \cong \RHom_{\GrMod A}(M(-1), A) \cong  \RHom_{\Lambda}(M, C)$. 
 Combining these isomorphisms, we obtain an isomorphism $\grRHom_{A}(M, A(1)) \cong \RHom_{\Lambda}(M,C)$ 
 as desired. 
\end{proof}

\begin{remark}
 In \cite{adasore},  
for a trivial extension algebra $A = \Lambda \oplus C$, 
we introduced the right  asid number $\alpha_{r}$ and the left asid number $\alpha_{\ell}$ 
which are defined by the following formulas  
 \[
 \alpha_{r}:= \max\{ a \geq  -1 \mid \grRHom_{\Lambda}(\Lambda, A)_{a} \neq 0\} + 1, 
 \alpha_{\ell}:= \max\{ a \geq  -1 \mid \grRHom_{\Lambda^{\op}}(\Lambda, A)_{a} \neq 0\} + 1.
 \]
It is clear that $\alpha_{r} = 0 =\alpha_{\ell}$ if and only if $A$ is hwg. 
\end{remark}

\section{The Happel functor and a homologically well-graded IG-algebra}\label{the Happel functor and a homologically well-graded IG-algebra}

In this Section \ref{the Happel functor and a homologically well-graded IG-algebra} we study a finitely graded IG-algebra from a view point of the Happel functor.

\subsection{When is the Happel  functor $\varpi$ fully faithful?}\label{fully faithful}

By Proposition \ref{201708231410}, 
if $A$ is right hwg, then the Happel  functor $\varpi: \sfD^{\mrb}(\mod^{[0,\ell-1]} A) \to \grSing A$ is fully faithful.
The aim of Section \ref{fully faithful} is to prove the converse under 
the assumption that 
$A$ is  IG and 
that the base ring $\kk$ is  Noetherian and $A= \bigoplus_{i= 0}^{\ell} A_{i}$ is a finitely generated as a $\kk$-module.  
We note that the latter  condition is equivalent to assume that $A_{i}$ is finitely generated over $\kk$ for $i = 0,1, \cdots, \ell$. 
We use the assumption only to establish the following lemma and corollary. 

\begin{lemma}\label{module finite lemma} 
Assume that $\kk$ is Noetherian and  a $\kk$-algebra $\Lambda$ is finitely generated  as a $\kk$-module. 
Let $D$ be a   bimodule over $\Lambda$ which is finitely generated on both sides, $M$  an object of $\sfD^{\mrb}(\mod \Lambda)$ and $n \in \ZZ$.   
If we regard  $\Hom_{\Lambda}(D, M[n])$ as a $\Lambda$-module by using the $\Lambda^{\op}$-module structure of $D$, 
then it is finitely generated.  
\end{lemma}

\begin{proof}
By standard argument, the problem is reduced to show that $\Ext_{\Lambda}^{n}(D, M)$ belongs to $\mod \Lambda$ for 
$n \in \ZZ$ and $M \in \mod \Lambda$. 
Taking a projective resolution $P^{\bullet}$ of $D$   as a $\Lambda$-module such that $P^{i}$ is finitely generated, 
we see that $\Ext_{\Lambda}^{n}(D, M) = \tuH^{n}(\Hom_{\Lambda}(P^{\bullet}, M))$ is finitely generated over $\kk$ with respect to the $\kk$-module structure 
induced from the $\Lambda$-module structure of $D$. 
Let $I^{\bullet}$ be an injective resolution of $M$. Then we have the following quasi-isomorphisms 
\[
\Hom_{\Lambda}(P^{\bullet}, M) \xrightarrow{ \sim} \Hom_{\Lambda}(P^{\bullet}, I^{\bullet}) \xleftarrow{\sim} \Hom_{\Lambda}(D, I^{\bullet}).
\]
of complexes of $\kk$-modules. 
Therefore, $\tuH^{n}(\Hom_{\Lambda}(D, I^{\bullet}))$ is finitely generated with respect to the $\kk$-module structure 
induced from the $\Lambda$-module structure of $D$. 
Since a bimodule $D$ is assumed to be $\kk$-central, 
$\tuH^{n}(\Hom_{\Lambda}(D, I^{\bullet}))$ is finitely generated with respect to the $\kk$-module structure 
induced from the $\Lambda^{\op}$-module structure of $D$. 
Thus we conclude that $\Ext_{\Lambda}^{n}(D, M)$ is finitely generated as a $\Lambda$-module. 
\end{proof}

\begin{corollary}\label{202003121955}
Let $A = \bigoplus_{i = 0}^{\ell}A_{i}$ be a finitely graded Noetherian algebra. 
Assume that $\kk$ is Noetherian and $A$ is a  finitely generated $\kk$-module.
Let  $M$ be an object of $\sfD^{\mrb}(\mod \Lambda)$. 
Then the object $\grRHom_{\Lambda}(A, M) \in \sfD(\GrMod A)$ belongs to $\sfD(\grmod A)$.
Moreover, the object  $\grRHom_{\Lambda}(A, M) \in \sfD(\GrMod A)$ is of bounded cohomology 
if and only if so are $\RHom_{\Lambda}(A_{i}, M) \in \sfD(\mod \Lambda)$ for all $i = 0, 1, \cdots, \ell$. 
\end{corollary}

The following is the main result of Section \ref{fully faithful}.

\begin{theorem}\label{main theorem 2}
Let $A = \bigoplus_{i = 0}^{\ell}A_{i}$ be a finitely graded IG-algebra. 
Assume that $\kk$ is Noetherian and $A$ is a  finitely generated $\kk$-module.
 Then the following conditions are  equivalent. 

\begin{enumerate}[(1)]
\item $A$ is a hwg algebra. 

\item The functor $\varpi $  is fully faithful. 

\item $\Ker \varpi = 0$.  

\item
We have $\Hom_{\grSing A}(\varpi(T), \varpi(T) [n]) = 0$ for $n \neq 0$.  
The algebra homomorphism  $\gamma: \nabla A  \to \End_{\grSing A} (\varpi T) 
$ induced from the functor $\varpi$ is an isomorphism. 
\end{enumerate}
\end{theorem}

We remark that the algebra homomorphism $\gamma$ in Theorem \ref{main theorem 2} coincides with 
the algebra homomorphism $\gamma$ of  \eqref{Introduction morphism gamma} 
if we identify the functor $\cH$ with $\varpi$ via the equivalence $\beta: \stabgrCM A \cong \grSing A$. 

We need a preparation.

\begin{lemma}\label{finite cohomology lemma}
Assume that $\kk$ is Noetherian. 
Let $\Lambda$ be an algebra which is finitely generated  as a $\kk$-module, 
$C$  a   bimodule over $\Lambda$ which is finitely generated on both sides 
and $A = \Lambda \oplus C$ the trivial extension algebra with the grading $\deg \Lambda := 0, \ \deg C: =1$. 
Let $I \in \sfC^{\mrb}(\grInj A)$ such that (the quasi-isomorphism class of) it belongs to $\sfD^{\mrb}(\grmod A)$.  
Then, the following assertions hold. 
\begin{enumerate}[(1)]
\item  The complex $\frki_{i} I$ belongs to $\sfD^{\mrb}(\mod \Lambda)$ for $i \in \ZZ$. 
 \item The complex $\frks_{\geq i} I$   belongs to $\sfD^{\mrb}(\grmod A)$ for $i \in \ZZ$.  
 \item The complex $\frks_{< i}I$  belongs to $\sfD^{\mrb}(\grmod A)$ for $i \in \ZZ$.  
 \end{enumerate}
\end{lemma}

\begin{proof}
We may assume that $I \neq 0 $ in $\sfD(\GrMod A)$. 
Since $I$ belongs to $\sfD^{\mrb}(\grmod A)$, 
the subset $\{i \in \ZZ \mid I_{i} \neq 0  \textup{ in } \sfD(\Mod \Lambda) \}$ of $\ZZ$ is non-empty and bounded.  
We set $j := \max\{i \in \ZZ \mid I_{i} \neq 0  \textup{ in } \sfD(\Mod \Lambda) \}$.

(1) 
By Lemma \ref{20170819245}(1),
$\frki_{i}I = 0$ for $i > j$. 
It is only remained to prove  that $\frki_{i} I$ belongs to $\sfD^{\mrb}(\mod \Lambda)$ for $i\leq j$.
We check this by descending induction on $i$. 

First, we deal with the case where $i=j$. 
Since $\frki_{j} I$ is isomorphic to  $I_{j}$ in $\sfD(\Mod \Lambda)$  by Lemma \ref{20170819245}(3),   
it belongs to  in $\sfD^{\mrb}(\mod \Lambda)$.

Next, let $i \leq j$. We assume that $\frki_{i} I$ belongs to $\sfD^{\mrb}(\mod \Lambda)$. 
Since $\frki_{i} I$ is a bounded complex of injective modules,  
using Lemma \ref{module finite lemma}, we see that $\RHom_{\Lambda}(C, \frki_{i} I)$ belongs to $\sfD^{\mrb}(\mod \Lambda)$.  
Using an exact triangle 
$\frki_{i-1}  I \to I_{i-1} \to \RHom_{\Lambda}(C, \frki_{i} I) \to $ in $\sfD^{\mrb}(\mod \Lambda)$ obtained in \cite[Lemma 5.1]{adasore},  we deduce that $\frki_{i -1} I$ belongs to $\sfD^{\mrb}(\mod \Lambda)$.

(2) 
By Lemma \ref{20170819245}(2),  
we have $\frks_{\geq i} I = 0$  in $\sfD^{\mrb}(\grmod A)$ for $i > j$. 
It is only remained to prove  that $\frks_{\geq i} I$ belongs to $\sfD^{\mrb}(\grmod A)$ for $i \leq j $.
We check this by descending induction on $i$. 

We note that it follows from (1) and Corollary \ref{202003121955}  that 
the object $\frks_{i} I = \grRHom_{\Lambda}(A, \frki_{i}I)( -i)$ belongs to $\sfD^{\mrb}(\grmod A)$ for $i \in \ZZ$.

First, we deal with the case where $i=j$. 
Since $\frks_{\geq j} I$ is isomorphic to  $\frks_{j} I$ in $\sfD(\GrMod A)$  by Lemma \ref{20170819245}(2),   
it belongs to  in $\sfD^{\mrb}(\grmod A)$.
 
 Next, let $i\leq j$. We assume that $\frks_{\geq i} I$ belongs to $\sfD^{\mrb}(\grmod A)$. 
 Observe that $\frks_{ \leq i -1} (\frks_{ \geq i -1 } I) = \frks_{ i -1} I$. 
 Replacing $i$ and $I$  with $i -1$ and $\frks_{\geq i-1 } I$ in the exact triangle \eqref{exact triangle}, 
 we obtain an exact triangle $\frks_{i-1} I \to \frks_{  \geq i-1} I \to \frks_{\geq i } I \to$.  
Using this exact triangle and the induction hypothesis,  
we see that $\frks_{\geq i -1 } I$ belongs to $\sfD^{\mrb}(\grmod A)$. 

(3) is proved from the assumption on $I$ and (2) by using the exact triangle \eqref{exact triangle}. 
\end{proof}

We proceed a proof of Theorem \ref{main theorem 2}. 

\begin{proof}[Proof of Theorem \ref{main theorem 2}]
By Lemma \ref{201708232231} and  the diagram \eqref{qv diagram}, we may assume that $A = \Lambda \oplus C$. 
The implications (1) $\Rightarrow$ (2) follows from Proposition \ref{201708231410}. 
The implication (2) $\Rightarrow$ (3) is clear. 

We prove (3) $\Rightarrow$ (1). 
Let $I$ be an injective resolution of $A$. We claim that $\frks_{\leq 0}I$ belongs to $\sfD^{\mrb}(\mod^{0} A)
= \sfD^{\mrb}(\mod \Lambda)$. 
Indeed, 
by Lemma \ref{finite cohomology lemma}, $\frks_{\leq 0} I$ belongs to $\sfD^{\mrb}(\mod^{\leq 0} A)$. 
By Lemma \ref{20170819245} we have an isomorphism $\frks_{> 0} I \cong \frks_{1} I$ in $\sfD(\GrMod A)$. 
Therefore $I\cong A$ and $\frks_{> 0} I$ belong to $\sfD(\mod^{[0,1]}A)$. 
It follows from the exact triangle  
$\frks_{\leq 0} I \to I \xrightarrow{f} \frks_{ > 0} I \to $ of  \eqref{exact triangle} 
that $\frks_{\leq 0} I$ belongs to $\sfD^{\mrb}(\mod^{[0,1]} A)$. 
Hence we conclude that $\frks_{\leq 0} I$ belongs to 
$\sfD^{\mrb}(\mod^{\leq 0} A) \cap \sfD^{\mrb}(\mod^{[0,1]}A) = \sfD^{\mrb}(\mod^{0}A)$ 
as desired. 

By 
Lemma \ref{derived interpretation of Iwanaga lemma},  
 $\frks_{\leq 0} I$ belongs to $\sfK^{\mrb}(\grproj A)$. 
Therefore  $\varpi(\frks_{\leq 0}I) = 0$. 
It follows from the assumption $\Ker \varpi = 0$  that $\frks_{\leq 0} I = 0$.  
Thus, $I \cong \frks_{1} I$ and $A$ is right hwg by Lemma \ref{201708191905}.  
This finishes the proof of the implication (3) $\Rightarrow$ (1). 

The implication (2) $\Rightarrow$ (4) is clear. 
We prove the implication (4) $\Rightarrow$ (3). 
It follows from the assumption that the restriction of $\varpi$ gives an equivalence 
\[\varpi|_{\sfK^{\mrb}(\proj \Lambda)} : \sfK^{\mrb}(\proj \Lambda) = \thick \Lambda \xrightarrow{\sim} \thick \varpi \Lambda.
\]
In particular, we have $(\Ker \varpi) \cap \sfK^{\mrb}(\proj \Lambda) = \Ker( \varpi|_{\sfK^{\mrb}(\proj \Lambda)}) = 0$. 
On the other hand, by \cite[Theorem 4.17]{adasore}, we have $\Ker \varpi \subset \sfK^{\mrb}(\proj \Lambda)$. 
Therefore,  we conclude $\Ker \varpi = 0$.  
\end{proof}

\subsection{When does the Happel  functor $\varpi$ give an equivalence?}\label{equivalence?}

In Section \ref{equivalence?}, we discuss when the functor $\varpi$ gives an equivalence.

\subsubsection{The condition (F)} 

First we introduce a finiteness condition on homological dimensions. 

\begin{definition}\label{definition: finite local dimension}
An algebra $\Lambda$ is said to satisfy the condition (PF) (resp. (IF) ) 
if all finitely generated $\Lambda$-module $M$ satisfies 
$\pd M < \infty$ (resp. $\injdim M < \infty$).  

An algebra $\Lambda$ is said to satisfy the condition (F) if it satisfies both the conditions (PF) and (IF). 
\end{definition}

It is clear that if $\gldim \Lambda < \infty$, then $\Lambda$ satisfies the condition (F). 
In some  cases,  the converse holds. 
In the case where $\Lambda$ is a finite dimensional algebra, $\Lambda$ satisfies the condition (F) if and only if it is of finite global dimension. 
More generally, in the case where $\kk$ is a complete local Noetherian  ring and $\Lambda$ is finitely generated as $\kk$-module, 
then $\Lambda$ is a semi-perfect Noetherian algebra (see e.g. \cite[Proposition 6.5 and Theorem 6.7]{Curtis-Reiner}). 
It follows that the condition (F) implies $\gldim \Lambda < \infty$. 

We collect  basic properties of the condition (F). 

\begin{lemma}\label{lemma: finite local dimension}
Assume that  $\Lambda$ is Noetherian. 
Then, the following assertions hold. 

\begin{enumerate}[(1)]
\item 
$\Lambda$ satisfies the condition (PF) (resp. (IF))  
if and only if $\sfK^{\mrb}(\proj \Lambda) = \sfD^{\mrb}(\mod \Lambda)$ 
(resp. $\sfD^{\mrb}(\mod \Lambda) \subset \sfK^{\mrb}(\Inj \Lambda)$). 
 
\item 
$\Lambda$ satisfies the condition (F) 
if and only if $\sfK^{\mrb}(\proj \Lambda) = \sfD^{\mrb}(\mod \Lambda) \subset \sfK^{\mrb}(\Inj \Lambda)$. 

\item 
The following conditions are equivalent. 

\begin{enumerate}[(a)]
\item $\Lambda$ and $\Lambda^{\op}$ satisfy the condition (F). 

\item $\Lambda$ is IG and satisfies the condition (F). 

\item $\Lambda$ satisfies the condition (F) and has a cotilting bimodule $C$. 

\item $\Lambda$ and $\Lambda^{\op}$ satisfy the condition (IF). 
\end{enumerate}
\end{enumerate}
\end{lemma}

\begin{proof}
We leave the proofs of (1), (2) to the readers. 

We prove (3). 
The implication $(a) \Rightarrow (b)$ is clear. 
The implication $(b) \Rightarrow (c)$ is proved by setting $C = \Lambda$.

$(c) \Rightarrow (d)$. 
It is enough to show that $\Lambda^{\op}$ satisfies the condition (IF). 
By Theorem \ref{cotilting theorem}, the functor $F: = \RHom_{\Lambda}(-, C): \sfD^{\mrb}(\mod \Lambda)^{\op} \to \sfD^{\mrb}(\mod \Lambda^{\op})$ is an equivalence. 
On the other hand, by (2), we have $\sfD^{\mrb}(\mod \Lambda) = \sfK^{\mrb}(\proj \Lambda) = \thick \Lambda$. 
It follows from an isomorphism $F(\Lambda) \cong C_{\Lambda^{\op}}$ that 
$\sfD^{\mrb}(\mod \Lambda^{\op}) = \thick C_{\Lambda^{\op}}$.  
By the definition of a cotilting module, we have $\injdim C_{\Lambda^{\op}} < \infty$. In other words, $C \in \sfK^{\mrb}(\Inj \Lambda^{\op})$. 
Thus we conclude 
$\sfD(\mod \Lambda^{\op}) = \thick C_{\Lambda^{\op}} \subset \sfK^{\mrb}(\Inj \Lambda^{\op})$. 

$(d) \Rightarrow (a)$. 
It follows from (1) that $\sfD^{\mrb}(\mod \Lambda) \subset \sfK^{\mrb}(\Inj \Lambda)$.  
It is easy to see that $\Lambda$ is IG. 
Applying ungraded version of Lemma \ref{derived interpretation of Iwanaga lemma},  we have 
$\sfK^{\mrb}(\proj \Lambda) = \sfD^{\mrb}(\mod \Lambda) \cap \sfK^{\mrb}(\Inj \Lambda) = \sfD^{\mrb}(\mod \Lambda)$. 
By (1), $\Lambda$ satisfies the condition (PF). 

In the same way, we can show that $\Lambda^{\op}$ satisfies the condition (PF).   
\end{proof}

%

To use the quasi-Veronese algebra construction, we need the following lemma. 
Since it is easily proved by using \cite[Proposition 6.1]{adasore}, we leave the proof to the readers. 

\begin{lemma}
Let $\Gamma_{1}, \Gamma_{2}$ be Noetherian algebras and $E$ be a $\Gamma_{1}$-$\Gamma_{2}$-bimodule 
which is finitely generated on both sides. Then the upper triangular matrix algebra $\Gamma = \begin{pmatrix} \Gamma_{1} & E \\ 0 & \Gamma_{2} \end{pmatrix}$ satisfies the condition (PF) (resp. (IF), (F))  if and only if so do $\Gamma_{1}$ and $\Gamma_{2}$. 
\end{lemma} 

\begin{corollary}\label{upper triangular lemma}
Let $A = \bigoplus_{i = 0}^{\ell} A_{i}$ be a finitely graded algebra. 
Then $A_{0}$ satisfies the condition (PF) (resp. (IF), (F) ) if and only if so does $\nabla A$. 
\end{corollary}

\subsubsection{Existence of a generator in $\grSing A$}\label{generator} 

In this Section \ref{generator} we discuss relationship between 
existence of generators in $\grSing A$  and the condition (F) on $A_{0}$. 

Let $\sfT$ be a triangulated category. 
An object $S \in \sfT$ is called a \emph{thick generator} if $\thick S = \sfT$. 
An object $S \in \sfT$ is called a \emph{generator} (resp. \emph{cogenerator}) if 
for an object $X \in \sfT$ the condition that 
$\Hom_{\sfT}(S, X[n]) = 0$ (resp. $\Hom_{\sfT}(X, S[n]) = 0$) for $n \in \ZZ$ 
implies $X= 0$. 
It is easy to see that a thick generator is both a  generator and a  cogenerator. 
We note that a tilting object $S \in \sfT$ defined in Section \ref{tri cat} is a thick generator.

The following proposition in the case of graded self-injective algebra  over a field is shown in \cite[Lemma 3.5]{Yamaura}, 
and the later generalized in \cite[Lemma 3.2]{LZ} 
where the self-injective assumption is dropped. 

%

\begin{proposition}\label{thick varpi proposition}
Assume that $\kk$ is Noetherian and $A= \bigoplus_{i= 0}^{\ell}A_{i}$ is an  IG-algebra which is finitely generated as a $\kk$-module. 
If  $A_{0}$ satisfies the condition (F), then  $\varpi(T)$ is a thick generator of $\grSing A$. 
\end{proposition}

\begin{proof}
We may assume that $A = \Lambda \oplus C$ by Corollary \ref{upper triangular lemma}. 
We note that $T =\Lambda$.

Observe that every  object $X \in \sfD^{\mrb}(\grmod A)$  is constructed from $M(n)$ for $M \in \sfD^{\mrb}(\mod \Lambda)$ and $n \in \ZZ$ by taking extensions iteratively. 
Therefore every  object $ \pi X \in \grSing A $ is constructed from $\pi M(n)$ for $M \in  \sfD^{\mrb}(\mod \Lambda)$ and $n \in \ZZ$ by 
taking extensions iteratively. 
Thus it is enough to check that $\pi M(n) \in \thick \varpi \Lambda $  for $M \in \sfD^{\mrb}(\mod \Lambda)$ and $n \in \ZZ$.

We fix $M \in \sfD^{\mrb}(\mod \Lambda)$ and prove that $\pi M(n) \in \thick \varpi \Lambda$  for $n \in \ZZ$.

First, we deal with the case where $n = 0$. 
We have  $M \in \sfK^{\mrb}(\proj \Lambda)=\thick \Lambda$ by Lemma \ref{lemma: finite local dimension}(2). 
Therefore, $\pi M \in \thick \varpi \Lambda$. 

 Next, we deal with the case $n >0$. For simplicity, we set $C^{n} := C\lotimes_{\Lambda} \cdots \lotimes_{\Lambda} C$ ($n$-times) for $n > 0$. 
 It follows from $M \in \sfK^{\mrb}(\proj \Lambda)$ that  the complex $M \lotimes_{\Lambda} A$ belongs to $\sfK^{\mrb}(\grproj A)$. 
 From the exact triangle below, we see that $M(1)$ becomes isomorphic to $M\lotimes_{\Lambda} C[1]$  in $\grSing A$.
\[
M \lotimes_{\Lambda}C \to M \lotimes_{\Lambda} A(1) \to M(1) \to M \lotimes_{\Lambda}C[1]
\]
Therefore,  $M(n)$ become  isomorphic to $M \lotimes_{\Lambda} C^{n}[n]$ in $\grSing A$ for $n > 0$. 
Since $\sfD^{\mrb}(\mod \Lambda) = \sfK^{\mrb}(\proj \Lambda)$, 
the functor $- \lotimes_{\Lambda} C$ sends $\sfD^{\mrb}(\mod \Lambda)$ to $\sfD^{\mrb}(\mod \Lambda)$. 
It follows that $M \lotimes_{\Lambda} C^{n}[n]$ belongs to $\sfK^{\mrb}(\proj \Lambda) = \thick \Lambda$, 
Thus we conclude that  $\pi M(n)$ belongs to $\thick \varpi \Lambda$ for $n > 0$.

Finally, we deal with the case $n < 0$. 
We have $M \in \sfK^{\mrb}(\Inj \Lambda)$ by Lemma \ref{lemma: finite local dimension} (2), 
so it follows that  the complex $\grRHom_{\Lambda}(A, M)$ belongs to $\sfK^{\mrb}(\grInj A)$. 
Moreover, it also belongs to $\sfD^{\mrb}(\mod\Lambda)$ by Corollary \ref{202003121955}. 
Hence it belongs to $\sfK^{\mrb}(\grproj A)$ by Lemma \ref{derived interpretation of Iwanaga lemma}. 
We set $F(-):= \RHom_{\Lambda}(C, -)$. 
Then,  from an exact triangle below, we see that $M(-1)$ becomes  isomorphic to $F(M)[-1] =\RHom_{\Lambda}(C, M)[-1]$ in $\grSing A$. 
\[
\RHom_{\Lambda}(C, M)[-1] \to M(-1) \to \grRHom_{\Lambda}(A,M)(-1) \to \RHom_{\Lambda}(C, M)
\]
Therefore, we see that $M(n)$ becomes  isomorphic to $F^{-n}(M)[n]$ in $\grSing A$ for $n < 0$. 

By Corollary \ref{202003121955}, 
the functor $F(-) = \RHom_{\Lambda}(C, -)$ sends $\sfD^{\mrb}(\mod \Lambda)$ to $\sfD^{\mrb}(\mod \Lambda)$. 
It follows that the object $F^{-n}(M)[n]$ belongs to $\sfD^{\mrb}(\mod \Lambda) = \thick \Lambda$ for $n < 0$. 
Thus,  we conclude that  $\pi M(n)$ belongs to $\thick \varpi \Lambda$ for $n <0$.
\end{proof}

The above proposition asserts that  finiteness of homological dimension on $A_{0}$ implies existence of a thick generator  in $\grSing A$. 
We do not  know whether the converse holds or not. 
However,  a stronger generating condition, existence of a tilting object  implies that $A_{0}$ satisfies the condition $(F)$.

\begin{proposition}\label{tilting condition (F) proposition}
Assume that $\kk$ is Noetherian and $A= \bigoplus_{i = 0}^{\ell} A_{i}$ is an 
IG-algebra 
which is finitely generated as a $\kk$-module. 
If $\grSing A$ has a tilting object $S$, then $A_{0}$ and $A_{0}^{\op}$ satisfies the condition (F).
\end{proposition}

This proposition in the case of graded self-injective algebra over a field 
is shown in  \cite[Theorem 3.1]{Yamaura}, 
and later generalized in \cite[Lemma 3.1]{LZ} where the self-injective assumption is dropped.

We need some preparations first. 

\begin{lemma}\label{seiri lemma 1}
Assume that $\kk$ is Noetherian and $A= \Lambda \oplus C$ is an 
IG-algebra 
which is finitely generated as a $\kk$-module. 
Then the following assertions hold. 
\begin{enumerate}[(1)]
\item  We have $\pi \sfD^{\mrb}(\mod^{\leq 0} A ) \perp \pi\left( \sfD^{\mrb}(\grmod A) \cap \sfK^{+}(\Inj^{> 0 \textup{-cog}}A) \right)$. 

\item Assume moreover that  for all $M \in \sfD^{\mrb}(\mod \Lambda)$, we have
 $\RHom_{\Lambda}(C, M) \in \sfD^{\mrb}(\mod \Lambda)$.
Then, we have 
\[
\grSing A = \pi \sfD^{\mrb}(\mod^{\leq 0} A ) * \pi\left( \sfD^{\mrb}(\grmod A) \cap \sfK^{+}(\Inj^{> 0 \textup{-cog}}A) \right).
\]
\end{enumerate}
\end{lemma}

\begin{proof}
(1) 
First we remark  that since $\Hom_{\GrMod A}(M , I ) = 0$ for $M \in \mod^{\leq 0} A$ and $ I \in \Inj^{>0\textup{-cog} } A$, 
we have 
 $ \sfD^{\mrb}(\mod^{\leq 0} A ) \perp \left( \sfD^{\mrb}(\grmod A) \cap \sfK^{+}(\Inj^{> 0 \textup{-cog}}A) \right)$ in $\sfD^{\mrb}(\grmod A)$. 
 
Let $M \in \sfD^{\mrb}(\mod^{\leq 0} A)$ and 
 $ I \in  \sfD^{\mrb}(\grmod A) \cap \sfK^{+}(\Inj^{> 0\textup{-cog}}A)$. 
 We may assume that $I$ is represented by $I \in \sfC^{+}(\Inj^{> 0\textup{-cog}} A)$. 
Let $f \in \Hom_{\grSing A}(\pi(M), \pi (I))$. 
We take a  diagram below in $\sfD^{\mrb}(\grmod A)$ representing $f$.   
\[
M \xrightarrow{ f' } I' \xleftarrow{ s } I.  
\]
Namely,  $I'$ is an object of  $\sfD^{\mrb}(\grmod A)$ and 
$f': M \to I'$ and $s: I \to I'$ are morphisms  in $\sfD(\GrMod A)$ 
which satisfies the following properties: 
(i) The cone $\cone(s)$ of $s$ belongs to $\sfK^{\mrb}(\grproj A)$. (ii) The map $\pi(s)$ is invertible. (iii) We have $f= \pi(s)^{-1} \pi(f')$ 
in $\grSing A$. 

By Lemma \ref{derived interpretation of Iwanaga lemma}, there exists  $K \in \sfC^{\mrb}(\grInj A)$ which is quasi-isomorphic to $\cone (s)[-1]$. 
The canonical morphism $t: \cone(s)[-1] \to I$ is represented by a morphism $t: K \to I$ in $\sfC(\grInj A)$, which is denoted by the same symbol $t$. Since $\Hom_{\sfC(\GrMod A)}(\frks_{\leq 0}K, I ) = 0$ by Corollary \ref{202003111555}, there exists a morphism $\hat{t} : \frks_{>0} K \to I$ which complete the upper square of the following diagram.
\[
\begin{xymatrix}{
& K \ar[d]_{t} \ar[r] & \frks_{> 0} K \ar[d]^{\hat{t}} \\ 
& I \ar[d]_{s} \ar@{=}[r] & I \ar[d]^{\hat{s}}  \\
M \ar[r]_{f'}  & I' \ar[r]_{u} &I''  
}\end{xymatrix}
\] 
Let $I'' $ be the cone of $\hat{t} $ and $\hat{s} :I \to I''$ be a canonical morphism. 
Then there exists a morphism $u: I' \to I''$ such that $us = \hat{s}$. 

We claim that the morphism $\pi(\hat{s})$ is invertible. 
Indeed, first note that to prove the claim it is enough to show that  $\frks_{> 0}K$ belongs to $\sfK^{\mrb}(\grproj A)$.  
By the property (i), the object $K \cong \cone(s)[-1]$ belongs to $\sfK^{\mrb}(\grproj A)$. 
Hence it belongs to $\sfD^{\mrb}(\grmod A)$. 
It follows from Lemma \ref{finite cohomology lemma} that 
 $\frks_{> 0} K$ belongs to $\sfD^{\mrb}(\grmod A)$. 
Therefore  $\frks_{> 0} K$ belongs to $\sfD^{\mrb}(\grmod A) \cap \sfK^{\mrb}(\GrInj A)$. 
It follows from  Lemma \ref{derived interpretation of Iwanaga lemma} that $\frks_{> 0} K$  belongs to $\sfK^{\mrb}(\grproj A)$.  
Thus we conclude that the morphism $\pi(\hat{s})$ is invertible as desired.

By the above claim, we have $ f= \pi(s)^{-1} \pi(f') = \pi(\hat{s})^{-1}\pi(u f')$. 
On the other hand, 
since $\frks_{>0} K$ and $I$ belongs to $\sfK^{+}( \Inj^{> 0\textup{-cog}} A)$, 
so is the cone $I''$ of $\hat{t}$. 
Therefore $\Hom_{\sfD^{\mrb}(\grmod A)}(M, I'') = 0$ by the remark at the beginning of the proof. 
Thus, in particular $uf' = 0$ and consequently $f= 0$ as desired.

(2) 
Let $M \in \sfD^{\mrb}(\grmod A)$ and $I \in \sfC^{+}(\Inj A)$ an injective resolution of $M$. 
Using the same argument with the proof of Lemma \ref{finite cohomology lemma}, 
we verify that $\frks_{> 0} I, \frks_{\leq 0} I$ belong to $\sfD^{\mrb}(\grmod A)$. 
Thus in particular, $\frks_{> 0} I$ belongs to  $\sfD^{\mrb}(\grmod A) \cap \sfK^{+}(\Inj^{> 0\textup{-cog}}A)$. 
Applying $\pi$ to the exact triangle \eqref{exact triangle}, we obtain an exact triangle 
$\pi(\frks_{\leq 0} I ) \to \pi(M) \to \pi(\frks_{> 0} I) \to $, from which we deduce the desired conclusion. 
\end{proof}

\begin{corollary}\label{seiri corollary 1}
Assume that $\kk$ is Noetherian and $A= \Lambda \oplus C$ is an 
IG-algebra 
which is finitely generated as a $\kk$-module. 
Assume moreover that  for all $M \in \sfD^{\mrb}(\mod \Lambda)$, we have
 $\RHom_{\Lambda}(C, M) \in \sfD^{\mrb}(\mod \Lambda)$. 
 If $\grSing A$ has a  generator $S$, then the following assertions hold. 
 \begin{enumerate}[(1)] 
 
  \item We have $\sfD^{\mrb}(\grmod A) \cap \sfK^{+}(\Inj^{> 0\textup{-cog}}A)  \subset \sfK^{\mrb}(\grproj A)$. 
 \item We have 
 $ \grSing A = \pi \sfD^{\mrb}(\mod^{\leq 0} A ) $.

 \item $\Lambda$ satisfies the condition (IF). 
 \end{enumerate}
 \end{corollary}

\begin{proof}
(1) 
Let $\tilde{S}$ be an object of $\sfD^{\mrb}(\grmod A)$ such that $\pi(\tilde{S}) =S$. 
There exists an integer $i \in \ZZ$ such that $\tilde{S}_{ > i}=0$. 
Then the object $\tilde{S}(i)$ belongs to $\sfD^{\mrb}(\mod^{\leq 0}A)$. 
Therefore $S(i) = \pi(\tilde{S}(i))$ is a generator of $\grSing A$ which belongs to 
$\pi \sfD^{\mrb}(\mod^{\leq 0} A )$. 

By Lemma \ref{seiri lemma 1}, 
we have the equality  
$\pi\left( \sfD^{\mrb}(\grmod A) \cap \sfK^{+}(\Inj^{> 0\textup{-cog}}A) \right)  = 0$, which implies the desired result.

(2) follows from (1) and Lemma \ref{seiri lemma 1}. 

(3) 
Let $M \in \sfD^{\mrb}(\mod \Lambda)$ and $J \in \sfC^{+}(\Inj \Lambda)$ an injective resolution of $M$. 
By the assumption and Corollary \ref{202003121955}, $X := \grRHom_{\Lambda}(A, M)( -1)$ belongs to $\sfD^{\mrb}(\grmod A)$. 
On the other hand,  it is clear that $X \cong \cpxgrHom_{\Lambda}(A, J)(-1)$ belongs to $\sfK^{+}(\Inj^{> 0\textup{-cog}} A)$. 
It follows from (1)  that $\grRHom_{\Lambda}(A, M)$ belongs to $\sfK^{\mrb}(\grproj A)$.  
Since $A$ is IG, we have $\sfK^{\mrb}(\grproj A) \subset \sfK^{\mrb}(\grInj A)$. 
Therefore  $\grRHom_{\Lambda}(A, M)$ belongs to $\sfK^{\mrb}(\GrInj A)$.

We set $I := \cpxgrHom_{\Lambda}(A, J)$. Note that $I \in \sfC(\GrInj A)$ is an injective resolution of $\grRHom_{\Lambda}(A,M)$. 
Therefore it is homotopic to a bounded complex $ I' \in \sfC^{\mrb}(\GrInj A)$.
Hence $\frki_{0}I$ is homotopic to a bounded complex $\frki_{0}I' \in \sfC^{\mrb}(\Inj \Lambda)$.

It follows from Section \ref{202003151619} that $I_{> 0} \cong \grRHom_{\Lambda}(A, M)_{> 0}= 0$ in $\sfD(\GrMod A)$. 
Therefore we have $I_{0} \cong \frki_{0} I$ by Lemma \ref{20170819245}(1). 
On the other hand, 
it follows from Section \ref{202003151619} that $I_{0} \cong \grRHom_{\Lambda}(A, M)_{0} \cong  M$. 
This shows that  $M \cong \frki_{0}I$ belongs to  $\sfK^{\mrb}(\Inj \Lambda)$. 

Thus we conclude that $\sfD^{\mrb}(\mod \Lambda) \subset \sfK^{\mrb}(\Inj \Lambda)$ as desired. 
\end{proof}

\begin{corollary}\label{seiri corollary 2}
Assume that $\kk$ is Noetherian. 
Let $A = \bigoplus_{i = 0}^{\ell}A_{i}$ be a finitely graded IG-algebra which is  finitely generated as a $\kk$-module. 
We assume that $\pd_{A_{0}} A < \infty$ 
and that  $\grSing A$ has a generator. 
Then $A_{0}$ satisfies the condition (IF).  
\end{corollary}

\begin{proof}
It is clear that the $\ell$-th quasi Veronese algebra $A^{[\ell]} =  \nabla A \oplus \Delta A$ is finitely graded as a $\kk$-module. 
As is explained in Section \ref{Quasi-Veronese algebra construction}, $A^{[\ell]}$ is IG. 
It follows from \cite[Proposition 6.1]{adasore} that $\pd_{\nabla A} \Delta A < \infty$. 
Using Lemma \ref{module finite lemma}, 
we see that for any $M \in \sfD^{\mrb}(\mod \nabla A)$,  the object  $\RHom_{\nabla A}(\Delta A, M)$ belongs to $\sfD^{\mrb}(\mod \nabla A)$.  
This shows that $A^{[\ell]} =  \nabla A \oplus \Delta A$ with the grading $\deg \nabla A := 0, \ \deg \Delta A := 1$ satisfies the assumptions of 
Corollary \ref{seiri corollary 1}. 
It follows that $\nabla A$ satisfies the condition (IF). 
By Corollary \ref{upper triangular lemma}, we conclude that $A_{0}$ satisfies the condition (IF).
\end{proof}

We proceed the proof of Proposition \ref{tilting condition (F) proposition}.

\begin{proof}[Proof of Proposition \ref{tilting condition (F) proposition}]
We may assume that $A = \Lambda \oplus C$ by Corollary \ref{upper triangular lemma}. 
First we claim that $\Lambda$ satisfies the condition (IF). 
Since a tilting object $S$ is a generator, 
it follows from Corollary \ref{seiri corollary 1} (3) that we only have to show that 
for  any $M \in \mod \Lambda $, the complex   $\RHom_{\Lambda}(C, M)$ belongs to $\sfD^{\mrb}(\mod \Lambda)$. 
By Lemma \ref{module finite lemma}, it is enough to show that $\Ext^{n}_{\Lambda}(C,M) = 0$ for $n \gg 0$. 
We set $d := \injdim A$ and take an exact sequence 
\[
 0 \to D \to P^{-d+1} \to \cdots \to P^{0} \to C \to 0 
\]
in $\grmod A$ such that each $P^{-i}$ is finitely generated projective over $A$. 
The graded $A$-module $D$ is CM, since it is the $d$-th syzygy of $C$. 
Then,  for $n > 0$ we have the following isomorphisms 
\[
\begin{split}
\Ext_{\Lambda}^{n+ d} (C, M) & \cong \Ext_{\GrMod A}^{n+d }(C, M) \\ 
& \cong \Ext_{\GrMod A}^{n}(D, M) \\
& \cong \Hom_{\grSing A}(\pi D, (\varpi M) [n]).  
\end{split}
\]
The last map is an isomorphism, since  $D$ is CM. 
Since $\grSing A$ has a tilting object, we have $\Hom_{\grSing A}(\pi D, (\varpi M) [n]) = 0$ for $|n| \gg 0$ by \cite[Proposition 2.4]{Aihara-Iyama}. 
This finishes the proof of the claim. 

Since $\grSing A^{\op}$ is contravariantly  equivalent to $\grSing A$ by the $A$-duality $\RHom(-, A)$, 
it also has a tilting object. 
Thus, applying the first claim to $A^{\op}$ we see that $\Lambda^{\op}$ satisfies the condition (IF).
Thus by Lemma \ref{lemma: finite local dimension}, $\Lambda$ and $\Lambda^{\op}$ satisfies the condition (F). 
\end{proof}

We collect the following two results for a finitely graded Noetherian algebra which is not necessary IG. 
The proofs are left to the readers, since these can be done by the dual arguments of that of Lemma \ref{seiri lemma 1}, Corollary \ref{seiri corollary 1} 
and Corollary \ref{seiri corollary 2}. 

\begin{lemma}\label{seiri lemma 2}
Let $A = \Lambda \oplus C$ be a finitely graded Noetherian algebra. 
Assume  that  for $M \in \sfD^{\mrb}(\mod \Lambda)$, we have
 $M \lotimes_{\Lambda}C \in \sfD^{\mrb}(\mod \Lambda)$ and 
that $\grSing A$ has a  cogenerator S. 
Then the following assertions hold. 
 \begin{enumerate}[(1)] 
  \item We have $\sfK^{-, \mrb}(\proj^{< 0 \textup{-gen}}A)  \subset \sfK^{\mrb}(\grproj A)$.  
 \item We have 
 $ \grSing A = \pi \sfD^{\mrb}(\mod^{\geq 0} A ) $. 
 
 \item $\Lambda$ satisfies the condition (PF). 
 \end{enumerate}
\end{lemma}

\begin{proposition}\label{seiri proposition 2}
Let $A = \bigoplus_{i= 0}^{\ell} A_{i}$ be a finitely graded Noetherian algebra. 
Assume that $\pd_{A_{0}} A < \infty$. 
If $\grSing A$ has a cogenerator, then $A_{0}$ satisfies the condition (PF).  
\end{proposition}

\subsubsection{The condition for which $\varpi$ is  an equivalence}

We give equivalent  conditions for which  $\varpi$ is an equivalence. 

\begin{theorem}\label{tilting theorem} 
Assume that $\kk$ is Noetherian and $A= \bigoplus_{i = 0}^{\ell} A_{i}$ is an 
IG-algebra 
which is finitely generated as a $\kk$-module. 
Then the following conditions are equivalent.
\begin{enumerate}[(1)]
\item 
$A$ is hwg and $A_{0}$ satisfies the condition (F). 

\item 
The Happel  functor $\varpi$ induces an equivalence.

\item 
The object $\varpi(T)$ is a tilting object of $\grSing A$ and 
the induced map $\End_{\grmod A}(T) \to \End_{\grSing A} (\varpi T) $ is an isomorphism.

\item
$A$ is hwg and  $\grSing A$ has a tilting object.

\end{enumerate}
\end{theorem}

\begin{remark}
%
%
Lu and Zhu  showed in \cite[Proposition 3.4]{LZ} that 
for a finite dimensional graded IG-algebra $A$ such that $A_0$ is of finite global dimension, 
the module $T$ becomes  a tilting object in $\Sing^{\mathbb{Z}}A$  provided  that it is  a CM-module.
If $A$ is hwg, then $T$ is a CM-module by Theorem \ref{main theorem}. Thus, a part of the implication (1) $\Rightarrow$ (3) follows from 
their result.   
\end{remark}

\begin{proof}
By Corollary \ref{upper triangular lemma}, we may assume that $A = \Lambda \oplus C$ and hence $T =\Lambda$. 

The implication 
(1) $\Rightarrow$ (3) follows from Theorem \ref{main theorem 2} and Proposition \ref{thick varpi proposition}.

(3) $\Rightarrow$ (2). 
It is enough to show that $\varpi$ is essentially surjective. 
Since $ \varpi(\Lambda)$ is tilting object, we have $\grSing A = \thick \varpi(\Lambda)$. 
On the other hand, from the second assumption  we deduce that the restriction of $\varpi$ gives an equivalence $\sfK^{\mrb}(\proj \Lambda) = \thick \Lambda \to \thick \varpi (\Lambda)$. 
Thus we conclude that $\varpi$ is essentially surjective as desired.  

(2) $\Rightarrow$ (1). 
First we claim that $\Lambda$ satisfies the condition (IF). 
Since $\Lambda$ is a  generator of $\sfD^{\mrb}(\mod \Lambda)$, 
the object $\varpi(\Lambda)$ is a  generator of $\grSing A$. 
Thus, by Corollary \ref{seiri corollary 1}, it is enough to show that 
for any $M \in \sfD^{\mrb}(\mod \Lambda)$, we have 
$\RHom_{\Lambda}(C, M) \in \sfD^{\mrb}(\mod \Lambda)$. 
By Lemma \ref{module finite lemma}, it is enough to show that $\Hom_{\Lambda}(C,M[n]) = 0$ for $|n| \gg 0$. 
From the canonical exact sequence $ 0 \to C \to A(1) \to \Lambda( 1) \to 0$ in $\grmod A$ and $\pi A(1) = 0$,  
we deduce an isomorphism  $\varpi C \cong (\varpi \Lambda) (1)[-1]$ in $\grSing A$.  
We set $\widetilde{M} := \varpi^{-1}(\varpi(M)(-1))$.  
Then, 
\[
\begin{split}
\Hom_{\Lambda}(C, M[n] )  
& = \Hom_{\grSing A}(\varpi(C), \varpi(M)[n]) \\ 
&= \Hom_{\grSing A}(\varpi (\Lambda )(1)[ -1], \varpi(M)[n]) \\
& = \Hom_{\grSing A}(\varpi (\Lambda), \varpi(M)(-1)[n +1]) \\
& =\Hom_{\Lambda}(\Lambda, \widetilde{M}[n+1]) = \tuH^{n +1}(\widetilde{M}).
\end{split}\]
Since $\widetilde{M}$ belongs to $\sfD^{\mrb}(\mod \Lambda)$, 
we conclude that $\Hom_{\Lambda}(C, M[n]) = 0$ for $|n| \gg 0$. 
This finishes the proof of the claim.

It follows from Theorem  \ref{main theorem 2} and Theorem \ref{main theorem} that $C$ is a  cotilting bimodule over $\Lambda$. 
Therefore by Lemma \ref{compatibility of dualities}, the Happel functor associated to $A^{\op}$  is also an equivalence.  
Thus, we can apply the claim to $A^{\op}$ and deduce that 
$\Lambda^{\op}$ satisfies the condition (IF). 
Thus by Lemma \ref{lemma: finite local dimension}, $\Lambda$ satisfies the condition (F).

The implication (4) $\Rightarrow$ (1) follows from Proposition \ref{tilting condition (F) proposition}. 
Finally if we assume that the condition (1) is satisfied, then we already know that $\varpi(\Lambda)$ is a tilting object in $\grSing A$. 
This prove the implication (1) $\Rightarrow$ (4). 
\end{proof}

\section{Examples and constructions}\label{Examples and constructions}

\subsection{Truncated tensor algebras}\label{subsection-truncation}

In Section \ref{subsection-truncation}, we give a sufficient conditions for a truncated  tensor algebra to be  hwg IG. 
For an algebra $\Lambda$ and a bimodule $E$, 
we denote by
\[
\sfT_{\Lambda}(E)=\Lambda \oplus E \oplus \left(E^{\otimes_{\Lambda}2}\right) \oplus \left(E^{\otimes_{\Lambda}3}\right) \oplus \cdots\cdots \oplus\left(E^{\otimes_{\Lambda}i}\right) \oplus \cdots\cdots
\]
the tensor algebra of $E$ over $\Lambda$. 
It has a structure of a graded algebra with the grading  $\deg \Lambda = 0, \deg E = 1$.

\begin{proposition}\label{tensor-IG}
Let $\Lambda$ be an IG-algebra, $C$ a cotilting bimodule over $\Lambda$ and $\ell$ a non-negative integer.
Assume that $C^{\lotimes_{\Lambda}i} \in \mod\Lambda$ for all $1\leq i \leq \ell$.
Then the truncated algebra
\[
A:=T_{\Lambda}(C)/T_{\Lambda}(C)_{\geq \ell+1}
\]
is an $\ell$-hwg IG-algebra.
\end{proposition}

\begin{proof}
It is obvious that $A$ is Noetherian and finitely graded.
We show that $A$ satisfies the conditions of Theorem \ref{main theorem} (4).

We set $C^{n} := C \lotimes_{\Lambda} \cdots \lotimes_{\Lambda} C$ ($n$-times). 
First we remark that it follows from  the ungraded version of  Lemma \ref{derived interpretation of Iwanaga lemma} that  if $C$ is regarded as a $\Lambda$-module, then it belongs to $\sfK^{\mrb}(\proj \Lambda)$. Thus if $M$ belongs to  $\sfK^{\mrb}(\proj \Lambda)$, then so dose $M \lotimes_{\Lambda}C$. 
In particular, we see $C^{n} \in \sfK^{\mrb}(\proj \Lambda)$ for $n \geq 0$ by using induction. 

It follows from the first remark that the canonical morphism below is an isomorphism for $n \geq 1$. 
\[C^{n-1} \lotimes_{\Lambda}\RHom_{\Lambda}(C, C) \to \RHom_{\Lambda}(C, C^{n})\]
Since $\RHom_{\Lambda}(C,C) \cong \Lambda$, we obtain an isomorphism $\RHom_{\Lambda}(C, C^{\ell}) \cong C^{\ell -1}$. Using adjunction as  below, we inductively obtain the following  isomorphism  for each  $1 \leq i \leq \ell$.
\[
\begin{split}
\RHom_{\Lambda}(C^{i}, C^{\ell}) & \cong \RHom_{\Lambda}(C, \RHom_{\Lambda}(C^{i-1}, C^{\ell})) \\
  & \cong \RHom_{\Lambda}(C, C^{\ell -i +1}) \cong C^{\ell -i}
 \end{split}
\]
This shows that $A$ satisfies the condition (4-ii) and (4-iii) of Theorem \ref{main theorem}.  

It only remains to check the condition (4-i). Namely  we only have to show that $A_{\ell} = C^{\ell}$ is a cotilting bimodule over $\Lambda$. 
We leave to the readers the verification that the isomorphism $\RHom_{\Lambda}(C^{\ell}, C^{\ell}) \cong C^{0} = \Lambda$ 
obtained above coincide with the canonical morphism. 
Since $C^{\ell} \in \sfK^{\mrb}(\proj \Lambda)$, we see  $\injdim_{\Lambda} C < \infty$ by the ungraded version of  Lemma \ref{derived interpretation of Iwanaga lemma}. 
This shows that the bimodule $A_{\ell}$ satisfies the defining conditions on the $\Lambda$-module structure in Definition \ref{definition of cotilting modules}. 
By a dual argument, we can check that  $A_{\ell}$ satisfies the remaining conditions of a cotilting bimodule.
\end{proof}

As an application, we study the tensor product $A=\Lambda\otimes_{\kk}\kk[x]/(x^{\ell+1})$. 

\begin{example}
Let $\Lambda$ be a Noetherian algebra and $\ell$ a natural number. 
We set 
\[
A = \Lambda \otimes_{\kk} \kk[x]/(x^{\ell +1}), \ \ \deg x= 1.  
\]
We point out an isomorphism $A \cong \sfT_{\Lambda}(\Lambda)/\sfT_{\Lambda}(\Lambda)_{ \ell +1}$ of graded $\kk$-algebras.

In the case where $\Lambda$ is a finite dimensional algebra, it is known that  $A$ is IG if and only if so is $\Lambda$. 
Moreover,   Cohen-Macaulay representation theory of $A$  has been studied by several researchers  (see e.g. \cite{GLS,Lu,RS,Ringel-Zhang}). 
In this example,  
applying our result, we prove the above characterization for $A$ to be IG in general setting.  
Moreover,  we recover the construction of a tilting object in $\grSing A$ and a triangle equivalence  given in \cite[Lemma 3.6]{Lu}.

Let $\Lambda$ be  a Notherian algebra which is not necessarily IG. 
Then it is easy to show  that $A$ and  $A^{\op}$ satisfy the condition (2) of  Proposition \ref{proposition hwg algebra}. 
Therefore $A$ is hwg.

Next, 
we claim that $A$ is IG if and only if so is $\Lambda$, regardless of whether $\Lambda$ is finite dimensional. 
Indeed as we mentioned before, if $\Lambda$ is IG, then $\Lambda$ is a cotilting bimodule over $\Lambda$. 
Since $\Lambda^{\lotimes_{\Lambda} i} $ belongs to $\mod \Lambda$ for $i \geq 0$, 
therefore $A$ is hwg IG by Proposition \ref{tensor-IG}. 
On the other hand,  
if we assume that $A$ is IG, then it is hwg IG. It follows from Theorem \ref{main theorem} that $A_{\ell} = \Lambda$ is  a cotilting bimodule over $\Lambda$.  Therefore $\Lambda$ is IG.  
This finishes the proof of the claim. 

It is easy to check that the Beilinson algebra $\nabla A$ is the $\ell\times \ell$-upper triangular matrix algebra of $\Lambda$. 
\[
\nabla A=\begin{pmatrix}
\Lambda & \Lambda & \cdots &\Lambda \\
 0 &\Lambda &\cdots &\Lambda \\
\vdots &  \vdots &  &\vdots  \\
0 & 0  &\cdots &\Lambda \\ 
\end{pmatrix}
\]
By Proposition \ref{201708231410}, the Happel functor $\varpi : \sfD^{\mrb}(\mod \nabla A) \to \grSing A$ is fully faithful. 
Assume that  $\kk$ is Noetherian and $\Lambda$ is an IG-algebra which is finitely generated as a $\kk$-module.  
Then $\varpi$ gives an equivalence if and only if  $\Lambda$ satisfies the condition (F), if and only if $T$ is a tilting object of $\grSing A$. 
In particular,  the condition  $\gldim \Lambda < \infty$ implies that 
 the Happel functor $\varpi$ gives an equivalence and $T$  is a tilting object. 
 This last assertion was proved by Lu \cite[Lemma 3.6]{Lu} in the case where $\kk$ is a field. 
\end{example}


Cotilting modules $C$ satisfying  the condition $C^{\lotimes_{\Lambda} i} \in \mod \Lambda$ of Proposition \ref{tensor-IG} arise in higher dimensional Auslander-Reiten Theory as   the bimodule $C := \Ext_{\Lambda}^{n}(\tuD(\Lambda), \Lambda)$ over an $n$-representation infinite algebra $\Lambda$. 
In this context, the tensor algebra $\sfT_{\Lambda}(C)$ is a generalization of usual preprojective algebra $\Pi(Q)$. 
Using Proposition \ref{tensor-IG}, we give partial generalizations of result about preprojective algebra $\Pi(Q)$ of non-Dynkin quiver.

\begin{example}
In this example, for simplicity,   the base field is assumed to be algebraically closed and a quiver $Q$ is assumed to be finite and acyclic. 
Let $n \geq 1$ be a positive integer. 
The notion of $n$-representation infinite ($n$-RI) algebra  was introduced by Herschend-Iyama-Oppermann \cite{HIO} 
as a generalization  of path algebras $\kk Q$ of infinite representation type from the view point of higher dimensional AR-theory.

A finite dimensional algebra $\Lambda$ is called $n$-RI if it is of finite global dimension and satisfies the following conditions.  We have $\Ext_{\Lambda}^{m}(\tuD(\Lambda), \Lambda) = 0$ 
except $m = n$ and the bimodule $C := \Ext_{\Lambda}^{n}(\tuD(\Lambda) ,\Lambda)$ satisfies the condition that 
$C^{\lotimes_{\Lambda} i} \in \mod \Lambda$ for all $i \geq 0$. 

Suppose  $\Lambda$ is $n$-RI. 
Then,  the bimodule $C := \Ext_{\Lambda}^{n}(\tuD(\Lambda) ,\Lambda)$ is  cotilting. 
To see this, first recall that cotilting bimodules over $\Lambda$ are precisely tilting bimodules over $\Lambda$, since $\gldim \Lambda < \infty$. 
Then,  observe that $C $ is quasi-isomorphic to $\RHom_{\Lambda}(\tuD(\Lambda), \Lambda)[n]$ and that  
the latter complex is the $\Lambda$-dual of  a (co)tilting bimodule $\tuD(\Lambda)$. Hence it is a (co)tilting bimodule. 
Thus the bimodule $C$ satisfies the conditions of Proposition \ref{tensor-IG}.

A path algebra $\kk Q$ of infinite representation type is $1$-RI (and the converse is also true up to Morita equivalence) 
and the tensor algebra $\sfT_{\kk Q}(C)$ is the preprojective algebra $\Pi(Q)$. 
Therefore, for an $n$-RI algebra $\Lambda$  the tensor algebra $\sfT_{\Lambda}(C)$ is a natural  generalization of a preprojective algebra $\Pi(Q)$ of the path algebra $\kk Q$. 
Hence it  is called the $(n+1)$\emph{-preprojective algebra} of $\Lambda$, denoted by  $\Pi(\Lambda)$ and plays a crucial role in higher AR-theory.
\[
\Pi(\Lambda) :=\sfT_{\Lambda}(C)
\]

By Proposition \ref{tensor-IG}, the truncated $(n+1)$-preprojective algebra $\Pi(\Lambda)_{\leq \ell} = \Pi(\Lambda)/ \Pi(\Lambda)_{\geq \ell +1}$ is hwg IG for $\ell > 0$. 
This is a partial generalization of a result by Buan-Iyama-Reiten-Scott \cite{BIRSc}.

Let $Q$ be a finite acyclic non-Dynkin quiver.  
They associated  a finite dimensional  factor algebra $\Pi(Q)_{w} := \Pi(Q)/I_{w}$  to an element $w$ of the Coxeter group $W_{Q}$ 
and showed that it is an IG-algebra. 
Let $c \in W_{Q}$ be a Coxeter element satisfying the condition of \cite[Definition 2.1]{Kimura-1}. 
If   $w = c^{\ell + 1}$ is a multiple of  $c$, then we have  $\Pi(Q)_{c^{\ell +1} } = \Pi(\kk Q)_{\leq \ell}$. 
Thus, in this case, our result recovers that of \cite{BIRSc}. 

Let $Q$ be a finite acyclic quiver.  
The graded and ungraded singular derived categories  of $ \Pi(Q)_{w}$ plays an important role in cluster theory. 
Kimura \cite{Kimura-1,Kimura-2} gave constructions of tilting objects in $\grSing \Pi(Q)_{w}$ when $w$ satisfies some assumptions.  
In the case where $Q$ is non-Dynkin and  $w = c^{\ell +1}$,  
since  $\gldim \kk Q \leq 1$, we can apply Theorem \ref{tilting theorem} 
to $\Pi(Q)_{c^{\ell +1} } = \Pi(\kk Q)_{\leq \ell}$ and obtain a tilting object $T$ of $\grSing \Pi(Q)_{c^{\ell + 1}}$. 
This tilting object coincides with the tilting object $M$ given in  \cite[Theorem 4.7]{Kimura-1}. 
\end{example}

\subsection{Veronese algebra}

Let $n> 0$ be a natural number. 
Recall that  the $n$-th Veronese subalgebra $A^{(n)}$ of a graded algebra $A = \bigoplus_{i \geq 0} A_{i}$ 
is defined to be the subalgebra of $A$ generated by $\{ A_{in}\mid i \in \NN\}$ with the grading $(A^{(n)})_{i} := A_{in}$. 

Before stating our result, 
it  deserves to remark that the same construction does not preserve IG-property in general. 

\begin{example}
We provide two examples. 

\begin{enumerate}[(1)]

\item
Assume that $\kk$ be a field and set  $A = \wedge \kk^{3} \times \kk[x]/(x^{5})$ where the left factor is the exterior algebra of a $3$-dimensional vector space with the grading $\deg \kk^{3} := 1$ and  the degree of $x$ is  set to be $1$. 
Then, $A$ is self-injective and in particular IG  with $\ell =4$. However, it is easy to see that $2$nd Veronese algebra $A^{(2)}$ is isomorphic to $ \sfT_{\kk}(\kk^{3})/\sfT_{\kk}(\kk^{3})_{ \geq 2} \times \kk[y]/(y^{3})$ and  is not IG.

\item 
Let $A$ be a finite dimensional graded algebra defined by a quiver
\[
\xymatrix{
1 \ar@(lu,ld)_{a} \ar@/^5pt/[r]^{b} & 2 \ar@/^5pt/[l]^{c}
}
\]
with relations $a^2=bc$, $ab=cb=ca=0$ and degrees $\deg a=1$, $\deg b= 0$, $\deg c= 2$.
It can be checked that $A$ is swg IG with the maximal degree $\ell = 2$. 
However, it can be also checked that the $2$nd Veronese algebra $A^{(2)}$ is not IG. 
\end{enumerate}
\end{example}

Contrary to this, the $n$-th Veronese subalgebra of   hwg IG-algebras   is again hwg IG provided that $n$ divides the maximal degree $\ell$. 

\begin{proposition}
If $A$ is an $(mn)$-hwg algebra, then so is the $n$-th Veronese subalgebra $A^{(n)}$.  
Moreover if $A$ is an $(mn)$-hwg IG-algebra, then so is $A^{(n)}$.
\end{proposition}

\begin{proof}
The assertions follow from Proposition \ref{proposition hwg algebra} and Theorem \ref{main theorem}.
\end{proof}

\subsection{Tensor products and Segre products}\label{subsection-tensor}

In this Section \ref{subsection-tensor},  we consider the tensor product of given two graded algebras.
In the rest of this subsection, for simplicity we assume that $\kk$ is a field
and  graded algebras $A,B$ are finite dimensional.

\subsubsection{Tensor product}
 
Let $H:=A\otimes_{\kk}B$ be the tensor product algebra of $A$ and $B$ with the grading 
$
H_k:= \bigoplus_{k = i + j} A_{i} \otimes_{\kk} B_{j}. 
$
We note that if we set the maximal degrees of $A$ and $B$ to be $\ell_{A}$ and $\ell_{B}$, then the maximal degree $\ell_{H}$ of $H$ is 
$\ell_{A} + \ell_{B}$. 
It is known that IG-property preserved by this construction. 
More precisely the following assertion holds. 

\begin{proposition}[{\cite[Proposition 2.2]{AR}}]\label{Proposition 2.2 of AR}
$H$ is IG if and only if so are $A$ and $B$.
\end{proposition}

We prove that hwg IG-property is also preserved by this construction.

\begin{proposition}\label{tensor-preserve-hwg}
In the above setting we have the following assertions. 
\begin{enumerate}[(1)]
\item 
 $A$ and $B$ are  right hwg if and only if so is  $H$.
\item 
 $A$ and $B$ are  hwg  IG if and only  if so is $H$.
\end{enumerate}
\end{proposition}

We leave the proof of the following lemmas to the readers. 

\begin{lemma}\label{bimod-iso}
Let $\Lambda,\Lambda'$ be finite dimensional algebras over a field $\kk$, 
$E,F$ be finitely generated $\Lambda$-modules and   
 $E',F'$ be finitely generated $\Lambda'$-modules. 
Then  
there is an isomorphism
\[
\RHom_{\Lambda\otimes_{\kk}\Lambda'}(E\otimes_{\kk}E',F\otimes_{\kk}F') \cong \RHom_{\Lambda}(E,F)\otimes_{\kk}\RHom_{\Lambda'}(E',F') 
\]
in $\mathsf{D}(\Mod\kk)$.
\end{lemma}

\begin{lemma}\label{tensor isomorphism lemma} 
Let $f: U \to V$ and $f' : U' \to V$ be morphisms in $\sfD^{\mrb}(\mod \kk)$. 
Then, the tensor product $f\otimes f': U \otimes_{\kk} U' \to V \otimes_{\kk} V'$ is an isomorphism if and only if 
so are $f$ and $f'$. 
\end{lemma}

As a consequence we deduce the following proposition concerning on cotilting bimodules.

\begin{proposition}\label{tensor-preserve-cotilt}
Let $\Lambda,\Lambda'$ be finite dimensional algebras over a field $\kk$, $C$ a  bimodule over $\Lambda$ and $C'$ a bimodule over $\Lambda'$.
Then 
$C$ and $C'$ are cotilting if and only if 
 $C\otimes_{\kk}C'$ is a cotilting bimodule over $\Lambda\otimes_{\kk}\Lambda'$.
\end{proposition}

\begin{proof}
For simplicity we set $\Lambda'' := \Lambda \otimes_{\kk} \Lambda'$ and $C'' :=  C \otimes_{\kk} C'$. 
Using the  same argument with \cite[Proposition 2.2]{AR}, we can prove that 
$\injdim_{\Lambda''} C'' < \infty  $ if and only if $\injdim_{\Lambda} C< \infty, \injdim_{\Lambda'}C'< \infty$. 

We denote by  $\psi: \Lambda \to \RHom_{\Lambda}(C, C)$ the canonical morphism. 
We also denote by $\psi'$ and $\psi''$ the canonical morphisms involving $C'$ and $C''$ respectively. 
Under the isomorphism  $\RHom_{\Lambda''}(C'', C'') \cong  \RHom_{\Lambda}(C,C)\otimes_{\kk}\RHom_{\Lambda'}(C',C')$
of  Lemma \ref{bimod-iso}, $\psi''$ corresponds to $\psi \otimes \psi'$. 
\[
\begin{xymatrix}{
\Lambda \otimes_{\kk} \Lambda' \ar[rr]^{\psi \otimes \psi' \ \ \ \ \ \ \ \ \ \ \ \ \ \ \ \ \ \ \ \ \ } \ar@{=}[d] && 
\RHom_{\Lambda}(C,C)\otimes_{\kk}\RHom_{\Lambda'}(C',C') \ar[d]^{\cong} \\
\Lambda'' \ar[rr]_{\psi''} && \RHom_{\Lambda''}(C'',C'')
}\end{xymatrix}
\]
Therefore, $\psi''$ is an isomorphism if and only if so are $\psi$ and $\psi'$ by Lemma \ref{tensor isomorphism lemma}. 

Since the same statements above are proved for the left module structures on $C, C'$ and $C''$, 
we see that $C''$ is cotilting precisely when $C$ and $C'$ are cotilting. 
\end{proof}

We are ready to prove Proposition \ref{tensor-preserve-hwg}. 

\begin{proof}[Proof of Proposition \ref{tensor-preserve-hwg}]
(1) 
We denote by  $\phi_{A}: A \to \grRHom_{A_{0}}(A, A_{\ell_{A}})(- \ell_{A})$ the canonical morphism, 
and likewise  by $\phi_{B}$ and $\phi_{H}$ the canonical morphisms for $B$ and $H$. 
Then from Lemma \ref{bimod-iso}, we obtain an isomorphism 
\[
\grRHom_{A_{0}}(A, A_{\ell_{A}})( -\ell_{A}) \otimes_{\kk} \grRHom_{B_{0}}(B, B_{\ell_{B}})(- \ell_{B})  
\cong \grRHom_{H_{0}}(H, H_{\ell_{H}})( -\ell_{H}), 
\]
under which $\phi_{H}$ corresponds to $\phi_{A} \otimes \phi_{B}$. 
Thus  it follows from 
  Lemma \ref{tensor isomorphism lemma} that $H$ satisfies the condition (2) of Proposition \ref{proposition hwg algebra}  if and only if so do $A$ and $B$.  
Thus the assertions follows. 

(2) 
Although the assertion follows by (1) and  Proposition \ref{Proposition 2.2 of AR}, we  provide  another proof.
It follows from  Theorem \ref{main theorem} that a finitely graded algebra $A$ is  hwg IG if and only if it is hwg and $A_{\ell_{A}}$ is cotilting.  
Thanks to  (1) and Proposition \ref{tensor-preserve-cotilt},  we can deduce the desired conclusion by checking the latter condition.  
\end{proof}

\subsubsection{Segre product} 
We recall  another product of two graded algebras $A$ and $B$, 
the Segre product $S$,   
which is defined to be $S:=\bigoplus_{i\geq0}A_i\otimes_{\kk}B_i$ with a natural multiplication. 
The grading of $S$ is defined to be $S_{i} := A_{i} \otimes B_{i}$. 

Before stating our result, we point out that even if $A$ and $B$ are IG of the same maximal degree $\ell$,  
the Segre product $S$ can fail to be IG.

\begin{example}
Let $\Lambda$ be a non-IG-algebra and $C$ be a cotilting bimodule over $\Lambda$ (e.g., $C = \tuD(\Lambda)$). 
We set $A : = \Lambda \oplus C$ with the canonical grading 
and  $B := \kk \times (\kk[x]/(x^{2}))$ with the grading $\deg x : =1$. 
Then,  since $S$ is isomorphic to  $\Lambda \times  A$, it is not IG.  
\end{example}

Contrary to this, the Segre product of two  hwg IG-algebras having the same maximal degree  is again hwg IG. 

\begin{proposition}\label{Segre-preserve-hwg}
Under the above setting, the following assertions hold.
\begin{enumerate}[(1)]
\item 
$A,B$ are right $\ell$-hwg if and only if $S$ is a right $\ell$-hwg algebra.
\item 
$A,B$ are $\ell$-hwg IG if and only if $S$ is an $\ell$-hwg IG-algebra.
\end{enumerate}
\end{proposition}

\begin{proof}
(1) 
We denote by  $\phi_{A}: A \to \grRHom_{A_{0}}(A, A_{\ell})(- \ell)$ the canonical morphism, and likewise 
 by $\phi_{B}$ and $\phi_{S}$ the canonical morphisms for $B$ and $S$. 
Then from Lemma \ref{bimod-iso}, for $i = 0, \cdots, \ell$ we obtain an isomorphism 
\[
\RHom_{A_{0}}(A_{i}, A_{\ell}) \otimes_{\kk} \RHom_{B_{0}}(B_{i}, B_{\ell} )  
\cong \RHom_{S_{0}}(S_{i}, S_{\ell}), 
\]
under which $\phi_{H, i}$ corresponds to $\phi_{A,i} \otimes \phi_{B,i}$. 
Thus  it follows from 
  Lemma \ref{tensor isomorphism lemma} that $S$ satisfies the condition (2) of Proposition \ref{proposition hwg algebra}  if and only if so do $A$ and $B$.  
Thus the assertions follows. 

(2)
The assertion can be proved by the same argument to the second proof of Proposition \ref{tensor-preserve-hwg} (2).
\end{proof}

\section{Commutative case}\label{commutative case}

Foxby \cite{Foxby} and Reiten \cite{Reiten} (see also \cite[3.7]{FFGR})
showed that 
if a local commutative graded algebra $A =A_{0} \oplus A_{1}$ is IG, 
then it is hwg. 
It is worth noting that 
in commutative ring theory, 
an IG-algebra is called a Gorenstein algebra. 
A cotilting module is called  a canonical module 
and it has alias such as  a dualizing module and a Gorenstein module of rank $1$. 

The aim of Section \ref{commutative case} is to generalize  
the result by Foxby and Reiten to   any   commutative  finitely graded algebras.

\begin{theorem}\label{commutative Gorenstein theorem}
A commutative local finitely graded Gorenstein algebra 
$A = \bigoplus_{i= 0}^{\ell} A_{i}$ is hwg. 
\end{theorem}

\def\Kdim{{\textup{Kdim}}}
\def\frkm{{\mathfrak{m}}}
\def\depth{{\textup{depth}}}

The symbol $\grSpec A$ denotes the set of graded prime ideals of $A$. 
Since $A$ is finitely graded, 
the ideal $A_{+} := \bigoplus_{i \geq 1} A_{i}$ is nilpotent 
and hence contained in every graded prime ideal. 
Therefore a graded prime ideal $\frkp$ of $A$ is of the form $\frkp = \frkp_{0} \oplus A_{+}$. 
We have $A/\frkp = A_{0}/\frkp_{0}$ and hence  $\frkp_{0}$ is a prime ideal of $A_{0}$.

We denote by $E_{A}(M)$ the injective envelope of a graded $A$-module $M$.

\begin{lemma}
For $\frkp \in \grSpec A$, we have 
$E_{A}(A/\frkp) \cong \grHom_{A_{0}}(A, E_{A_{0}}(A_{0}/\frkp_{0}))$
\end{lemma}

\begin{proof}
Since $\grHom_{A_{0}}(A, E_{A_{0}}(A_{0}/\frkp_{0}))$ 
is a graded injective module containing $A/\frkp = A_{0}/\frkp_{0}$ 
as an essential submodule, 
we conclude the desired result.
\end{proof}

We collect  graded versions of 
well-known results about structures of minimal injective resolutions. 
For this purpose, we recall the definition of the graded Bass number. 

\begin{definition}[The graded Bass number]
For  $\frkp \in \grSpec A$ and $M \in \GrMod A$,  
we set 
\[
\mu^{n}_{i}(\frkp, M)
 := 
\dim_{\kappa(\frkp)}
 \grExt_{A}^{n}(\kappa(\frkp), M_{\frkp})_{i}.
\]
\end{definition}

\begin{theorem}[{A graded version of \cite[Theorem 18.7]{Matsumura}}]\label{Bass injective decomposition theorem}
Let $M \in \GrMod A$ and $I^{\bullet}$ a minimal graded injective resolution of $M$. 
Then we have 
\[
\begin{split}
I^{n} 
&\cong \bigoplus_{\frkp \in \grSpec A, i \in \ZZ} E_{A}(A/\frkp)^{\oplus \mu^{ n }_{i} (\frkp, M)}(-i),\\ 
\frks_{i} I^{n} 
&\cong \bigoplus_{\frkp \in \grSpec A} E_{A}(A/\frkp)^{\oplus \mu^{ n }_{i} (\frkp, M)}(-i). 
\end{split}
\]
\end{theorem}

\begin{theorem}[{A graded version of \cite[Theorem 18.8]{Matsumura}}]\label{Bass Gorenstein theorem}
Let $A$
be a  commutative local finitely graded Gorenstein algebra. 
Then, 
\[
\begin{split}
\mu^{n}_{i}(\frkp, A)  &= 0 \textup{ for } n \neq \textup{ht} \frkp, i \in \ZZ\\
\sum_{i \in \ZZ} \mu^{\textup{ht} \frkp }_{i}(\frkp, A)  & = 1.
\end{split}
\]
\end{theorem}

We remark that 
for $M \in \grmod A$ and $\frkp \in \grSpec A$,  
$\mu^{n}_{i}(\frkp, M) \neq 0$ 
if and only if $\grExt_{A}^{n}(A/\frkp, M)_{i} \neq 0$. 
The  lemma below is   a graded version of  \cite[Lemma 18.3]{Matsumura} .

\begin{lemma}\label{graded one up lemma}
Let $M \in \grmod A$ and 
$\frkp \subset  \frkq$ graded prime ideals of $A$.  
Assume that the Krull dimension $\Kdim A_{\frkq}/\frkp_{\frkq} = 1$. 
If $\mu_{i}^{n}(\frkp, M) \neq 0$ 
for some $n \geq 0$ and $ i \in \ZZ$, 
then 
$\mu^{n+1}_{i}(\frkq, M)  \neq 0$. 
\end{lemma}

\begin{proof}[Proof of Theorem \ref{commutative Gorenstein theorem}]

Let $I$ be a minimal graded injective resolution of $A$.  We may regard $I$ as an object of $\sfC(\GrInj A)$. 
We prove $A$ is hwg by showing $\frks_{i}I = 0$ for $i \neq \ell$. 
By Theorem \ref{Bass injective decomposition theorem} 
and Theorem \ref{Bass Gorenstein theorem}  
it is enough to show that $\mu^{\textup{ht} \frkp}_{i}(\frkp, A) = 0$ for $i \neq \ell$. 
 We set $d:= \Kdim A$ to the Krull dimension of $A$. 
By the assumption we have 
the equalities 
\[
d= \Kdim A = \depth A = \injdim A = \grinjdim A
\] 
where the last equality is proved in  \cite[Proposition 2.11]{adasore}.

First, we claim that $\mu_{\ell}^{d}(\frkm, A) \neq 0$. 
Let $x$ be a homogeneous $A$-regular element. 
Since $A$ is finitely graded, we have $\deg x= 0$ 
and hence  $\mu^{d}_{\ell}(\frkm, A) = \mu^{d-1}_{\ell}(\frkm, A/x)$. 
Moreover the residue algebra $A/xA$ satisfies the assumptions of Theorem \ref{commutative Gorenstein theorem} 
by \cite[Proposition 3.1,19]{Bruns-Herzog}.
Therefore, 
since, by \cite[Proposition 1.5.11]{Bruns-Herzog}, 
there exists a $A$-regular sequence $x_{1} , \cdots, x_{d}$ 
consisting of homogeneous elements, the claim reduced to the case $d= 0$. 
In the case $d= 0$, 
then $A$ is isomorphic to a degree shift of $E_{A}(k)= \grHom_{A_{0}}(A, E_{A_{0}}(k))$. 
Since the functor $\Hom_{A_{0}}(-, E_{A_{0}}(k) )$ is faithful, 
we have \[
\max\{ i \in \ZZ \mid E_{A}(k)_{i}\neq 0 \} = 0, 
\min\{ i \in \ZZ \mid E_{A}(k)_{i}\neq 0 \} = -\ell.
\]
Thus, by comparing the degree of $A$ and $E_{A}(k)$, 
we deduce that $A \cong E_{A}(k)(-\ell)$. 
This proves the claim.

By Theorem \ref{Bass Gorenstein theorem}, we deduce $\mu_{i}^{d}(\frkm, A) = 0$ for $ i \neq \ell$ from the claim. 

Assume that  $\mu^{\textup{ht} \frkp}_{i}(\frkp, A) \neq 0$ for some $ i \neq \ell$. 
Then by Lemma \ref{graded one up lemma}, we have 
$\mu^{d}_{i}(\frkm, A) \neq 0$. 
This contradict to what we have proved. 
Thus, $\mu^{\textup{ht} \frkp}_{i}(\frkp, A) =0$ for $i \neq \ell$.
This completes the proof. 
\end{proof}

\section{Remark: graded derived Frobenius extensions}\label{Graded derived Frobenius extensions}

Recall that a Frobenius algebra $A$ is an algebra which possesses a symmetry 
that a regular module $A_{A}$ is isomorphic to the dual module $\Hom_{\kk}(A, \kk)$. 
\begin{equation}\label{Frobenius algebra} 
 \Hom_{\kk}(A,\kk) \cong A. 
\end{equation}
An important fact is that
this isomorphism (\ref{Frobenius algebra}) implies that $A$ is self-injective. 
We emphasize that Frobenius algebras have applications such as 
topological field theory (see e.g. \cite{Kock}) 
because of the symmetry (\ref{Frobenius algebra}).

We mention that there exists a generalization of 
Frobenius algebras defined by existence of a symmetry like (\ref{Frobenius algebra}), that is, \textit{Frobenius extensions}.
A Frobenius extension is  an algebra extension $\Lambda \subset A$ 
such that $A$ is finitely generated projective left $\Lambda$-module 
and that there exists an isomorphism of $\Lambda$-$A$-bimodules 
\begin{equation}\label{Frobenius extension} 
\Hom_{\Lambda}(A,\Lambda) \cong A. 
\end{equation}
Frobenius extensions are related to other areas and   have been studied by many researchers (see for example 
\cite{Kadison}). 
However, in general this isomorphism (\ref{Frobenius extension}) does not implies that $A$ is Frobenius or IG. 

We propose another generalization of Frobenius algebras, which turns out to be IG, characterized by existence of a symmetry.   
As is stated in Theorem \ref{cotilting theorem}, 
a cotilting bimodule induces a duality between the derived categories of $A$ and $A^{\op}$. 
Thus, it is natural to take an analogy of the isomorphism \eqref{Introduction graded Frobenius 2} by using a cotilting bimodule $C$.  

Let $\Lambda$ be a Noetherian algebra and $C$ a cotilting bimodule over $\Lambda$.  
An algebra extension $\Lambda \subset A$ is called \textit{derived Frobenius extension} with respect to a cotilting bimodule $C$ 
if there exists an isomorphism in the derived category $\sfD(\mod \Lambda^{\op} \otimes_{\kk}A)$.   
\begin{equation}\label{derived Frobenius extension} 
\RHom_{\Lambda}(A,C) \cong A. 
\end{equation}
In other words, $\Hom_{\Lambda}(A, C)  \cong A$ and $\Ext_{\Lambda}^{> 0}(A, C) = 0$. 
We can show that  the isomorphism (\ref{derived Frobenius extension}) implies that $A$ is IG in the same way as Theorem \ref{main theorem}.
\begin{proposition}\label{derived Frobenius extension is IG}
 If $\Lambda \subset A$ is a derived Frobenius extension, then $A$ is IG. 
\end{proposition}

In a similar way, as a generalization of a graded Frobenius algebra, 
we may define  a graded derived Frobenius extension.
\begin{definition}
Let $A= \bigoplus_{i = 0}^{\ell}A_{i}$ be a  finitely graded Noetherian algebra with $\Lambda := A_{0}$. 
Then $A$ is called \emph{a graded derived Frobenius extension} of $\Lambda$ 
if there exists an isomorphism 
\[
\hat{\alpha}: A \cong \grRHom_{\Lambda}(A, A_{\ell})(-\ell)
\]
in the derived category of $\Lambda$-$A$-bimodules. 
\end{definition}

We can rephrase Theorem \ref{main theorem} as follows: 
the graded algebra extension $A_{0} \subset A$ is a graded derived Frobenius extension 
if and only if $A$ is a hwg IG-algebra.

{H.M. Department of Mathematics and Information Sciences,
Faculty of Science / Graduate School of Science,
Osaka Prefecture University}

{minamoto@mi.s.osakafu-u.ac.jp} 

$ $

{K.Y. Graduate Faculty of Interdisciplinary Research, Faculty of Engineering
University of Yamanashi}

{kyamaura@yamanashi.ac.jp}


{\small 
\begin{thebibliography}{ABCD}



\bibitem{Aihara-Iyama}
 Aihara, Takuma; Iyama, Osamu, 
Silting mutation in triangulated categories, J. Lond. Math. Soc. (2)  85  (2012),  no. 3, 633-668. 
 

\bibitem{AR}
Auslander, Maurice; Reiten, Idun, 
Cohen-Macaulay and Gorenstein Artin algebras.  Representation theory of finite groups and finite-dimensional algebras (Bielefeld, 1991),  221-245. 

\bibitem{Avramov-Foxby}
Avramov, Luchezar L.; Foxby, Hans-Bj\o rn, 
Homological dimensions of unbounded complexes. 
J. Pure Appl. Algebra  71  (1991),  no. 2-3, 129-155.

\bibitem{BIRSc}
Buan, Aslak Bakke; Iyama, Osamu; Reiten, Idun; Scott, Jeanne;
Cluster structures for $2$-Calabi-Yau categories and unipotent groups,
Compos. Math. 145 (2009), no. 4, 1035--1079. 

\bibitem{Bruns-Herzog}
 Bruns, Winfried; Herzog, Jurgen,
 Cohen-Macaulay rings. Cambridge Studies in Advanced Mathematics, 39. Cambridge University Press, Cambridge, 1993.

\bibitem{Buchweitz} 
Buchweitz, Ragnar-Olaf,  
Maximal Cohen-Macaulay Modules and Tate-Cohomology Over Gorenstein Rings, 
unpublished manuscript available at https://tspace.library.utoronto.ca/handle/1807/16682

\bibitem{Chen trivial}
X.-W. Chen, Graded self-injective algebras ``are" trivial
extensions,  J. Algebra \textbf{322} (2009), 2601-2606

\bibitem{Curtis-Reiner}
 Curtis, Charles W.; Reiner, Irving, Methods of representation theory. Vol. I. 
With applications to finite groups and orders. Pure and Applied Mathematics. A Wiley-Interscience Publication. John Wiley $\&$ Sons, Inc., New York, 1981. xxi+819 pp. 


\bibitem{Foxby} 
Foxby, Hans-Bj{\o}rn,
Gorenstein modules and related modules. 
Math. Scand.  31  (1972), 267-284 (1973). 


\bibitem{FFGR}
Fossum, Robert; Foxby, Hans-Bj{\o}rn; Griffith, Phillip; Reiten, Idun, 
Minimal injective resolutions with applications to dualizing modules and Gorenstein modules. 
Inst. Hautes Etudes Sci. Publ. Math. No. 45 (1975), 193-215. 

\bibitem{FGR}
Fossum, Robert M.; Griffith, Phillip A.; Reiten, Idun, 
Trivial extensions of abelian categories. Homological algebra of trivial extensions of abelian categories with applications to ring theory,
Lecture Notes in Mathematics, Vol. 456. Springer-Verlag, Berlin-New York, 1975.

\bibitem{GLS}
Geiss, Christof; Leclerc, Bernard; Schr\"{o}er, Jan,
Quivers with relations for symmetrizable Cartan matrices I: Foundations,
Invent. Math. 209 (2017), no. 1, 61--158. 

\bibitem{Iwanaga}
Iwanaga, Yasuo, 
On rings with finite self-injective dimension. II. 
Tsukuba J. Math.  4  (1980), no. 1, 107-113. 


\bibitem{Happel book} 
Happel, Dieter, 
Triangulated Categories in the Representation Theory of
Finite-Dimensional Algebras. 
London Mathematical  Society  Lecture Notes Series 119.  
Cambridge: 
Cambrdge University Press, 1988.


\bibitem{Happel-1991}
Happel, Dieter,
Auslander-Reiten triangles in derived categories of finite-dimensional algebras,
Proc. Amer. Math. Soc. 112 (1991), no. 3, 641--648.



\bibitem{Happel}
Happel, Dieter, 
On Gorenstein algebras.  Representation theory of finite groups and finite-dimensional algebras (Bielefeld, 1991),  389-404, 
Progr. Math., 95, Birkhauser, Basel, 1991. 

\bibitem{HIO}
Herschend, Martin; Iyama, Osamu; Oppermann, Steffen,
$n$-representation infinite algebras,
Adv. Math. 252 (2014), 292--342. 

\bibitem{Iyama: ICM}
Iyama, Osamu,
Tilting Cohen-Macaulay representations,
to appear in the ICM 2018 proceedings.

\bibitem{IO}
Iyama, Osamu; Oppermann, Steffen,
Stable categories of higher preprojective algebras,
Adv. Math. 244 (2013), 23--68.


\bibitem{Kadison}
Kadison, Lars, 
New examples of Frobenius extensions. 
University Lecture Series, 14. American Mathematical Society, Providence, RI, 1999. 


\bibitem{Kimura-1}
Kimura, Yuta, 
Tilting theory of preprojective algebras and $c$-sortable elements,
J. Algebra 503 (2018), 186--221.


\bibitem{Kimura-2}
Kimura, Yuta, 
Tilting and cluster tilting for preprojective algebras and Coxeter groups,
 Int. Math. Res. Not. IMRN 2019, no. 18, 5597-5634.


\bibitem{Kock}
Kock, Joachim,
Frobenius algebras and 2D topological quantum field theories,
London Mathematical Society Student Texts, 59. Cambridge University Press, Cambridge, 2004.

\bibitem{Lu}
Lu, Ming,
Singularity categories of representations of quivers over local rings,
arXiv:1702.01367.


\bibitem{LZ}
Lu, Ming; Zhu, Bin,
Singularity categories of Gorenstein monomial algebras, 
arXiv:1708.00311.


\bibitem{Matsumura}
 Matsumura, Hideyuki, 
Commutative ring theory. 
Translated from the Japanese by M. Reid. 
Cambridge Studies in Advanced Mathematics, 8. Cambridge University Press, Cambridge, 1986. 



\bibitem{MM}
Minamoto, Hiroyuki; Mori, Izuru, The structure of AS-Gorenstein algebras. Adv. Math. 226 (2011), no. 5, 4061-4095.

\bibitem{adasore}
Minamoto, Hiroyuki; Yamaura, Kota, 
Homological dimension formulas for trivial extension algebras, 
J. Pure Appl. Algebra  224  (2020),  no. 8. 





\bibitem{higehaji}
Minamoto,  Hiroyuki; Yamaura, Kota,  
On finitely graded IG-algebras and the stable category their CM-modules, 
arXiv:1812.03746




%
%


\bibitem{Miyachi}
Miyachi, Jun-ichi, 
Duality for derived categories and cotilting bimodules. 
J. Algebra  185  (1996),  no. 2, 583-603. 


\bibitem{Mori B-construction}
Mori, Izuru, 
B-construction and C-construction. 
Comm. Algebra  41  (2013),  no. 6, 2071-2091. 


\bibitem{MU}
Mori, Izuru; Ueyama, Kenta, 
Stable categories of graded maximal Cohen-Macaulay modules over noncommutative quotient singularities, 
Adv. Math. 297 (2016), 54--92. 

\bibitem{NV:Graded and Filtered}
N\u{a}st\u{a}sescu, Constantin; Van Oystaeyen, F.,  
Graded and filtered rings and modules. Lecture Notes in Mathematics, 758. Springer, Berlin, 1979




\bibitem{Orlov}
Orlov, Dmitri, 
Derived categories of coherent sheaves and triangulated categories of singularities.  
Algebra, arithmetic, and geometry: in honor of Yu. I. Manin. Vol. II,  503-531, 
Progr. Math., 270, Birkhauser Boston, Inc., Boston, MA, 2009. 

\bibitem{Reiten}
Reiten, Idun, 
The converse to a theorem of Sharp on Gorenstein modules. 
Proc. Amer. Math. Soc.  32  (1972), 417-420.

 \bibitem{RS}
Ringel, Claus Michael; Schmidmeier, Markus,
Invariant subspaces of nilpotent linear operators. I,
J. Reine Angew. Math. 614 (2008), 1--52.

\bibitem{Ringel-Zhang}
Ringel, Claus Michael; Zhang, Pu, 
Representations of quivers over the algebra of dual numbers. 
J. Algebra  475  (2017), 327-360. 

\bibitem{Verdier} 
Verdier, Jean-Louis, 
Des categories derivees des categories abeliennes. 
 With a preface by Luc Illusie. Edited and with a note by Georges Maltsiniotis. Asterisque  No. 239  (1996), xii+253 pp. 


\bibitem{Weibel}
 Weibel, Charles A.,  
An introduction to homological algebra. 
Cambridge Studies in Advanced Mathematics, 38. 
Cambridge University Press, Cambridge, 1994. 

\bibitem{Yamaura} 
Yamaura, Kota, Realizing stable categories as derived categories. Adv. Math.  248  (2013), 784-819. 

\end{thebibliography}
}
\end{document}